\documentclass[preprint,11pt,intlimits]{elsarticle}

\usepackage{amssymb}
\usepackage{mathrsfs}
\usepackage{amsthm}
\usepackage{amsmath}
\usepackage{epstopdf}
\usepackage{epsfig}
\usepackage{xcolor}
\usepackage{graphicx}
\usepackage{subfigure}
\usepackage{longtable,booktabs}
\usepackage{lineno,hyperref}
\usepackage{threeparttable}

\newtheorem{theorem}{Theorem}[section]
\newtheorem{lemma}[theorem]{Lemma}
\newtheorem{corollary}[theorem]{Corollary}
\newtheorem{proposition}[theorem]{Proposition}

\newtheorem{definition}[theorem]{Definition}

\newcommand{\abs}[1]{\ensuremath{\lvert#1\rvert}}
\newcommand{\norm}[1]{\ensuremath{\lVert#1\rVert}}

\newcommand{\RR}{\ensuremath{\mathbb{R}}}
\newcommand{\NN}{\ensuremath{\mathbb{N}}}

\DeclareMathOperator{\minnorm}{m}
\DeclareMathOperator{\spanned}{span}

\DeclareMathOperator{\inte}{Int}

\begin{document}

\begin{frontmatter}



\title{Linearization and invariant manifolds on the carrying simplex for competitive maps\footnotemark{}\footnotetext{This is an accepted manuscript of a paper which is to appear in Journal of Differential Equations}\footnotemark{}\footnotetext{\copyright\ 2019. This manuscript version is made available under the CC-BY-NC-ND 4.0 license http://creativecommons.org/licenses/by-nc-nd/4.0/}}
\author[add1]{Janusz Mierczy\'{n}ski}\ead{janusz.mierczynski@pwr.edu.pl}
\author[add2]{Lei Niu}\ead{lei.niu@helsinki.fi}
\author[add3]{Alfonso Ruiz-Herrera}\ead{ruizalfonso@uniovi.es}
\address[add1]{Faculty of Pure and Applied Mathematics, Wroc{\l}aw University of Science and Technology, Wroc{\l}aw, Poland}
\address[add2]{Department of Mathematics and Statistics, University of Helsinki, Helsinki FI-00014, Finland}
\address[add3]{Departamento de Matem\'{a}ticas, Universidad de Oviedo, Oviedo, Spain}

\begin{abstract}
A result due to M.W. Hirsch states that most competitive maps admit a carrying simplex, i.e., an invariant hypersurface of codimension one which attracts all nontrivial orbits. The common approach in the study of these maps is to focus on the dynamical behavior on the carrying simplex. However, this manifold is normally non\nobreakdash-\hspace{0pt}smooth. Therefore, not every tool coming from Differential Geometry can be applied. In this paper we prove that the restriction of the map to the carrying simplex in a neighborhood of an interior fixed point is topologically conjugate to the restriction of the map to its pseudo\nobreakdash-\hspace{0pt}unstable manifold by an invariant foliation. This implies that the linearization techniques are applicable for studying the local dynamics of the interior fixed points on the carrying simplex. We further construct the stable and unstable  manifolds on the carrying simplex. Our results give partial responses  to Hirsch's problem regarding the smoothness of the carrying simplex. We discuss some applications  in classical models of population dynamics.
\end{abstract}

\begin{keyword}
Carrying simplex\sep invariant foliation \sep pseudo\nobreakdash-\hspace{0pt}stable manifold \sep pseudo\nobreakdash-\hspace{0pt}unstable manifold \sep linearization \sep invariant manifold



\end{keyword}

\end{frontmatter}



\section{Introduction}
Since the early work of Hirsch \cite{hirsch1988} and Smith \cite{smith1986}, it is well known that most competitive maps admit a  carrying simplex, that is, an invariant hypersurface of codimension one, such that every nontrivial orbit is attracted towards it; see \cite{ortega1998exclusion, wang2002uniqueness, diekmann2008carrying, hirsch2008existence, Ruiz-Herrera2013, Baigent2015, jiang2015, Jiang2017}. The importance of the  carrying simplex stems from the fact that it captures the relevant long-term dynamics.  In particular, all nontrivial fixed points, periodic orbits, invariant closed curves and heteroclinic cycles lie on the carrying simplex (see, for example, \cite{jiang2015, Jiang2016, Jiang2017, Gyllenberg2018a, Gyllenberg2018b,Gyllenberg2019}). In order to analyze the global dynamics of such discrete-time systems, it suffices to study the dynamics of the systems restricted to this invariant hypersurface. In particular,  one can use the topological results on the homeomorphisms of the plane such as the translation arc and degree (Ruiz-Herrera \cite{Ruiz-Herrera2013}, Jiang and Niu \cite{Jiang2016} and  Niu and Ruiz-Herrera \cite{nonlinearity2018}) for three\nobreakdash-\hspace{0pt}dimensional competitive maps with a carrying simplex.

In \cite{hirsch1988}, Hirsch posed the problem to determine conditions under which the carrying simplex is a smooth manifold (see \cite[P. 61]{hirsch1988}). This is a long-open question in dynamical systems with two direct applications. Obviously, the smoothness of the carrying simplex provides geometrical information on the manifold. On the other hand, and more importantly, the smoothness of the carrying simplex allows us to apply the tools coming from Differential Geometry, especially, the Grobman--Hartman theorem. To the best of our knowledge, the available results on the smoothness of the carrying simplex are the following: Jiang, Mierczy\'nski and Wang in \cite{Jiang2009} gave equivalent conditions, expressed in terms of inequalities between Lyapunov exponents, for the carrying simplex to be a $C^1$ submanifold-with-corners, neatly embedded in the nonnegative orthant (for sufficient conditions in the case of ordinary differential equations, see Brunovsk\'y \cite{Palo1994}, Mierczy\'nski \cite{Mierczynski1994}, or, for the $C^k$ property in discrete time systems, Bena\"{\i}m \cite{Benaim1997} and, in ordinary differential equations, Mierczy\'nski \cite{Mierczynski1999a}).  Mierczy\'nski proved in \cite{Mierczynski2018b,Mierczynski2018a} that the carrying simplex is a $C^1$ submanifold-with-corners neatly embedded in the nonnegative orthant when it is convex. For the convexity of the carrying simplex and their influence on the global dynamics, we refer the reader to \cite{Zeeman1994, Zeeman2003, Baigent2015, Baigent2017, Baigent2019}.  Whether the carrying simplex is smooth or not is still unknown when it is not convex. Mierczy\'nski in \cite{Mierczynski1999b, Mierczynski1999b} in the case of ordinary differential equations and Jiang, Mierczy\'nski and Wang in \cite{Jiang2009} in the case of maps do provide examples which show that the carrying simplex at a boundary fixed point can be far from smooth. However, no examples are known of the lack of smoothness in the interior of a carrying simplex. Roughly speaking, many competitive maps admit a reduction of the dimension but we do not know if this reduction is smooth.

This scenario suggests the following interesting questions: \emph{Can we use the ``linearization" techniques on the carrying simplex to study the local dynamics around a fixed point?  How should we construct the stable and unstable  manifolds even if the carrying simplex is not smooth?} It is well known that linearization techniques and invariant manifolds are important tools in the study of smooth dynamical systems (see, for example,  \cite{Pugh1969, Hirsch1977, Quandt1986, Zhang1993, Bronstein1994, Tan2000, Nipp2013, Kuznetsov3}). In this paper, we  prove that one can still use the ``linearization" techniques to study the dynamics on the carrying simplex even if it is non-smooth. Furthermore, we  construct the stable and unstable manifolds of an interior fixed points on the carrying simplex by those of the conjugate ``linear" term of the reduction.

The main tool of this paper consists in a topological result that guarantees the existence of an invariant foliation in a neighborhood of a fixed point when the inverse of its Jacobian matrix has strictly positive entries. This result (Theorem \ref{thm:foliation}) is deduced in Section \ref{sec:sec-inv-fol} and could be perceived not only as a technique for constructing the invariant manifolds on the carrying simplex but have its own interest. We refer the reader to \cite{Hirsch1977, Fenichel1979, Palis1993, Chow1991, Bates2000, zhang_zhang_2016, Zhang2016} for the discussion and application of invariant foliations in the study of dynamical systems. By using the previous invariant foliation, we prove in Section \ref{inv-maniflod} that the restriction of the map to the carrying simplex in a neighborhood of an interior fixed point is topologically conjugate to the restriction of the map to its pseudo-unstable manifold (Theorem \ref{tconjugacy}). This means that linearization techniques are applicable for studying the local dynamics on the carrying simplex because the restriction to its pseudo-unstable manifold is smooth. The consequence is that the invariant manifolds of the interior fixed points on the carrying simplex are homeomorphic to those of the restriction to its pseudo-unstable manifold (Theorem \ref{thm:inv-local-mani}). We will then prove the continuity of the tangent cones of the carrying simplex near the interior fixed points (Theorem \ref{thm:tangent}). Tangent cones also play remarkable roles in the study of global stability of the monotone dynamical systems (see \cite{Tineo2008, Zeeman2003}). In Section \ref{sec:sec-application}, we apply our results to some classical models in population dynamics that include the Leslie--Grower models, Atkinson--Allen models and Ricker models. In particular, we show that the stable manifold of the interior fixed point on the carrying simplex for three\nobreakdash-\hspace{0pt}dimensional competitive maps is indeed a simple curve when its index is $-1$, which solves an open problem in \cite{nonlinearity2018}. It is worth noting that many results of the paper can be applied to maps that admit a non-smooth center manifold.

\section{Notation and definitions}\label{sec:sec-notation}
\noindent Throughout this paper, we need the following notation and definitions. As usual, $\lVert \cdot \rVert$ stands for the Euclidean norm in $\mathbb{R}^n$, as well as for the operator norm with respect to the Euclidean norm.  For a linear automorphism $A \colon Z \to Z$, where $\{0\} \ne Z \subset \mathbb{R}^n$, we denote by $\minnorm(A)$ its \emph{co-norm},
\begin{equation*}
  \minnorm(A) := \min\{\, \lVert Au \rVert : u \in Z, \lVert u \rVert = 1 \,\}.
\end{equation*}

Let $\mathbb{R}^n_+ := \{x\in \mathbb{R}^n: x_i\geq 0 \mathrm{~for~all~} i=1,\ldots,n\}$ be the usual nonnegative orthant. The interior of $\mathbb{R}^n_+$ is the open cone $\inte{\mathbb{R}^n_+} := \{x\in \mathbb{R}^n_+ : x_i > 0$ for all $i = 1,\ldots,n\}$ and the boundary of $\mathbb{R}^n_+$ is $\partial \mathbb{R}^n_+ := \mathbb{R}^n_+ \setminus  \inte{\mathbb{R}^n_+}$.

For $x, y \in \mathbb{R}^n$, we write $x \le  y$ if $x_i \le  y_i$ for all $i=1,\ldots,n$, and $x \ll  y$ if $x_i < y_i$ for all $i=1,\ldots,n$. If $x \le  y$ but $x \ne  y$ we write $x <  y$. The reverse relations are
denoted by $\geq, >,\gg$, and so forth.

For a differentiable map $P$,  the Jacobian matrix of $P$ at the point $x$ is denoted by $DP(x)$.

\begin{definition}
A map $T \colon \mathbb{R}^{n}_{+} \to \mathbb{R}^{n}_{+}$ is {\em competitive} in a subset $W \subset \mathbb{R}^{n}_{+}$, if, for all $x, y \in W$ with $ T(x) < T(y)$, one has that $x_i < y_i$ provided $y_i > 0$.
\end{definition}

The carrying simplex for a map $T \colon \mathbb{R}^{n}_{+} \to \mathbb{R}^{n}_{+}$ is an invariant subset $S \subset \mathbb{R}^{n}_{+}$ with the following properties:
\begin{enumerate}[(H1)]
\item
    No two points in $S$ are related by the $<$ relation.
\item
    $S$ is homeomorphic via radial projection to the $(n-1)$\nobreakdash-\hspace{0pt}dimensional standard probability simplex $\Delta^{n-1} := \{x\in \mathbb{R}_+^n:\sum_{i=1}^n x_i=1\}$.
\item
    For any $x \in \mathbb{R}^{n}_{+}\setminus \{0\}$, there is some $y\in S$ such that $\displaystyle \lim_{m\to +\infty} \norm{T^m(x) - T^m(y)} = 0$.
\item
    $T(S) = S$ and $T|_S \colon S \to S$ is a homeomorphism.
\item
    $S$ is the boundary (relative to $\mathbb{R}^{n}_{+}$) of the global attractor $\Gamma$, which equals $\{\alpha x :  \alpha\in [0, 1], x \in S\}$. Moreover, $\Gamma \setminus S = \{\alpha x :  \alpha\in [0, 1), x \in S\}$ is the basin of repulsion of the origin.
\end{enumerate}

Most competitive maps, especially those used in population dynamics, admit a carrying simplex, which determines the dynamical behavior of the systems. The reader can consult \cite{ortega1998exclusion, wang2002uniqueness, diekmann2008carrying, hirsch2008existence, Ruiz-Herrera2013, Baigent2015, jiang2015, Jiang2017} for precise results on the existence of a carrying simplex in competitive maps.

\section{Invariant foliation in a neighborhood of a fixed point when the inverse of the Jacobian matrix is positive}\label{sec:sec-inv-fol}
Consider a map
 \begin{equation*}
 P \colon V \subset \mathbb{R}^{n} \to P(V) \subset \mathbb{R}^{n}
 \end{equation*}
 of class $C^{1}$ defined on an open neighborhood $V$ of $q\in\mathbb{R}^{n}$ with $P(q) = q$. We assume the following condition:
\begin{enumerate}
\item[\textbf{(C1)}]
    There exists  $(DP(q))^{-1}$ and its entries are strictly positive. Moreover, the eigenvalue of $DP(q)$ with the smallest modulus, say $\mu$, satisfies $0 < \mu < 1$.
\end{enumerate}
The classical Perron--Frobenius theorem guarantees that the first statement of (C1) implies that $\mu$ is always a simple positive eigenvalue. Moreover, the corresponding invariant subspace is spanned by some $v \gg 0$.  The invariant subspace $W$ of $\RR^n$ that corresponds to the remaining  eigenvalues of $DP(q)$ intersects $\mathbb{R}^{n}_{+}$ only at the origin.\newline Fix
\begin{equation*}
  \rho \in (\mu, \min \{1, \nu\})
\end{equation*}
with  $\nu$  the modulus of the eigenvalue(s) of $DP(q)$ with the second smallest modulus. The spectrum of  $DP(q)$ consists of two nonempty parts: one, consisting of a simple eigenvalue $\mu$, contained inside the circle centered at zero with radius $\rho$, and the one contained outside this circle. Note that $\det DP(q)\neq 0$, so, as in Sections~\ref{sec:sec-inv-fol} and~\ref{inv-maniflod} we are interested in the local behavior only, we can assume, without loss of generality, that $P$ is a diffeomorphism taking $V$ onto $P(V)$.

Now we present the main result of this section. In the statement of the theorem we employ the above notation.

\begin{theorem}\label{thm:foliation}
Assume that \textup{(C1)} holds. Then, fixed $\sigma \in (\rho, \nu)$, there exist a neighborhood $U$ of $q$ and the following objects:
\begin{enumerate}
  \item[\textup{(a)}]
  A one\nobreakdash-\hspace{0pt}dimensional $C^1$ manifold $M_1 \subset U$ that is tangent at $q$ to $v$. Moreover,
  \begin{equation*}
    \norm{P({\xi'}) - {P(\xi'')}} \leq \rho \norm{{\xi'} - {\xi''}} \quad \text{for each } {\xi', \xi''} \in M_1.
  \end{equation*}
  $M_1$ is \emph{positively invariant}, {\it i.e.} if $\xi \in M_1$, then $P(\xi) \in M_1$.
  \item[\textup{(b)}]
A one-codimensional $C^1$ manifold $M_2 \subset U$ that is tangent at $q$ to $W$. $M_2$ is \emph{locally invariant} in the sense that if $y \in M_2$ and $P(y) \in U$,  then we have that $P(y) \in M_2$. Analogously, if $y \in M_2$ and  $P^{-1}(y) \in U$, then we have that $P^{-1}(y) \in M_2$.  Moreover, there is $l \in \mathbb{N}$ so that
  \begin{equation}\label{eq:M2}
    \norm{P^{-l}({y'}) - {P^{-l}(y'')}} \leq \sigma^{-l} \norm{{y' - y''}}
  \end{equation}
  for any ${y', y''} \in M_2$ with $P^{-1}({y'}), \ldots, P^{-l}({y'}){, P^{-1}(y''), \ldots, P^{-l}(y'')} \in M_2$.

  \item[\textup{(c)}]
  A foliation $\mathcal{L}$ of $U$ by $C^1$ embedded segments $L_y$ \emph{(leaves)}, parameterized by $y \in M_2$ and linearly ordered by the $\ll$ relation.  The foliation $\mathcal{L}$ is \emph{locally invariant}. That is, for any $y\in M_{2}$ and $\xi\in L_{y}$, we have the following:
  \begin{itemize}
    \item  If $P(y) \in M_2$, then $P(\xi) \in L_{P(y)}$.
    \item  If $P^{-1}(y) \in M_2$ and $P^{-1}(\xi) \in U$, then $P^{-1}(\xi) \in L_{P^{-1}(y)}$.
  \end{itemize}
Moreover,
  \begin{equation*}
    \lVert P({\xi'}) - P({\xi''}) \rVert \leq \rho \lVert {\xi'} - {\xi''} \rVert
  \end{equation*}
  for any $y \in M_2$ and any ${\xi', \xi''} \in L_y$, provided that $P(y) \in M_2$.
\end{enumerate}
\end{theorem}

The goal of the rest of the section is to prove Theorem \ref{thm:foliation}. For simplicity in the notation,  we assume that the fixed point is the origin. Next we give several preliminary results.

\begin{lemma}
\label{lm:P-modified}
  For each $\epsilon > 0$, there exist $\eta > 0$ and a $C^1$ diffeomorphism $P_{\epsilon} \colon \mathbb{R}^n \to \mathbb{R}^n$ that satisfies
  \begin{equation*}
    P_{\epsilon}(\xi) =
    \begin{cases}
      P(\xi) & \text{for } \lVert \xi\rVert \le \tfrac{1}{2}\eta
      \\
      DP(0)\xi & \text{for } \lVert \xi \rVert \ge \eta
    \end{cases}
  \end{equation*}
  and
  \begin{equation*}
    \norm{DP_{\epsilon}(\xi) - DP(0)} < \epsilon, \quad \xi \in \RR^n.
  \end{equation*}
\end{lemma}
\begin{proof}
Take a  $C^{\infty}$ function $f \colon [0, \infty) \to [0,1]$ with the property that $f(r) = 1$ if and only if $r \in [0, 1/2]$ and $f(r) = 0$ if and only if $r \ge 1$. For each $\eta > 0$, we define
\begin{equation*}
  P_{\eta}(\xi) := DP(0)\xi + f\Bigl(\frac{\lVert \xi \rVert}{\eta}\Bigr) (P(\xi) - DP(0)\xi), \quad \xi \in \mathbb{R}^n.
\end{equation*}
  By~\cite[Thm.~2.1.7]{Hirsch1994}, the set of $C^1$ diffeomorphisms of $\mathbb{R}^n$ onto itself is open in the $C^1$ strong (Whitney) topology.  As the linear map
  \begin{equation*}
    \xi \mapsto DP(0)\xi
  \end{equation*}
  is a diffeomorphism,  there is a continuous function $\delta \colon \mathbb{R}^n \to (0, \infty)$ with the following property:  if $Q \colon \mathbb{R}^n \to \mathbb{R}^n$ is a map of class $C^1$ such that the difference between the $1$\nobreakdash-\hspace{0pt}jet of $Q$ and the 1\nobreakdash-\hspace{0pt}jet of $DP(0)$ at the point $\xi$ is smaller than $\delta(\xi)$ for all $\xi\in \mathbb{R}^n$, then $Q$ is a diffeomorphism.  The $1$\nobreakdash-\hspace{0pt}jets of $P_{\eta}$ and $DP(0)$ coincide for $\lVert \xi \rVert > \eta$. Hence it suffices to estimate the $C^1$ norm of
\begin{equation}\label{estimate}
    \xi \mapsto f\left(\frac{\lVert \xi \rVert}{\eta}\right) (P(\xi) - DP(0)\xi)
\end{equation}
 restricted to $\lVert \xi \rVert \le \eta$.  We know that $f$ takes values between $0$ and $1$ and
 \begin{equation*}
   \lVert P(\xi) - DP(0)\xi \rVert \to 0
 \end{equation*}
 as $\xi \to 0$. This implies that the $C^0$ norm of \eqref{estimate} tends to zero as $\eta \longrightarrow 0^{+}$. On the other hand,  $DP_{\eta}(\xi) - DP(0)$ is equal to
\begin{equation*}
   \frac{1}{\eta} f'\Bigl(\frac{\lVert \xi \rVert}{\eta}\Bigr) \frac{P(\xi) - DP(0)\xi}{\lVert \xi \rVert} \xi^{\top} + f\Bigl(\frac{\lVert \xi \rVert}{\eta}\Bigr) (DP(\xi) - DP(0)),\quad \xi \in \RR^n,
\end{equation*}
where $\xi^{\top}$ is the transpose of $\xi$. In the first summand, the norm of $f'(\frac{\lVert \xi \rVert}{\eta}) \frac{\xi^{\top}}{\eta}$ is bounded as $\lVert \xi \rVert \le \eta$.  Further, $\lVert P(\xi) - DP(0)\xi \rVert/\lVert \xi \rVert \to 0$ as $\xi \to 0$.  This implies that the first summand converges to $0$ as $\eta \to 0^+$.  The second summand converges to $0$ as $\eta \to 0^+$ as well.

Collecting all the information, we have proved that the difference between the $1$\nobreakdash-\hspace{0pt}jets of $P_{\eta}(\xi)$ and $DP(0)$ tends to $0$ as $\eta \longrightarrow 0^{+}$. For $\delta^{*} = \min \{\delta(\xi) : \Vert \xi \rVert \leq 1\}$,  we take $\eta_{0} \leq 1$ small enough so that the difference between the $1$\nobreakdash-\hspace{0pt}jets of $P_{\eta_{0}}$ and $DP(0)$ is smaller that $\min\{\epsilon,\delta^{*}\}$.  The map $P_{\eta_{0}}$ is the desired diffeomorphism $P_{\epsilon}$.
\end{proof}

In the sequel we will apply the results in~\cite{Hirsch1977} on invariant manifolds, invariant foliations, etc., which are formulated for small $C^1$ perturbations of linear maps.  It will be tacitly assumed that $\epsilon > 0$ is so small that a corresponding result in~\cite{Hirsch1977} can be applied to $P_{\epsilon}$ chosen from Lemma~\ref{lm:P-modified}.

\begin{proposition}
\label{prop:aux1}
For a suitable $\epsilon>0$, the map $P_{\epsilon}$ admits the following objects:
    \begin{enumerate}
    \item[\textup{(i)}]
    An invariant one-dimensional $C^1$ manifold $M_1$, tangent at $0$ to $v$, whose elements are characterized as those $\xi \in \mathbb{R}^n$ for which $\lVert P_{\epsilon}^k(\xi) \rVert/\rho^k$ stays bounded as $k \to +\infty$ \textup{[}or, equivalently, as those $\xi \in \mathbb{R}^n$ for which
    \begin{equation*}
      \lVert P_{\epsilon}^k(\xi) \rVert/\rho^k \to 0
    \end{equation*}
    as $k \to +\infty$\textup{]},

    \item[\textup{(ii)}]
    An invariant one-codimensional $C^1$ manifold $M_2$, tangent at $0$ to $W$, whose elements are characterized as those $y \in \mathbb{R}^n$ for which $\lVert P_{\epsilon}^{-k}(y) \rVert/{\rho^{-k}}$ stays bounded as $k \to +\infty$ \textup{[}or, equivalently, as those $y \in \mathbb{R}^n$ for which
    \begin{equation*}
      \lVert P_{\epsilon}^{-k}(y) \rVert/{\rho^{-k}} \to 0
    \end{equation*}
    as $k \to +\infty$\textup{]}.
  \end{enumerate}
\end{proposition}
\begin{proof}
  See \cite[Thm.~5.1]{Hirsch1977}.
\end{proof}

We denote by $\mathcal{T} M_2$ the tangent bundle of $M_2$.   For each  $y \in M_2$, $\mathcal{T}_{y} M_2$ is the tangent space of $M_2$ at $y$. From now on, we assume that $\epsilon > 0$ in the construction of $P_{\epsilon}$ is so small that $v$ is transversal to $\mathcal{T}_{y} M_2$ at each $y \in M_2$.

\begin{lemma}
\label{lm:normally-attracting}
  For $P_{\epsilon}$ with $\epsilon > 0$ sufficiently small the manifold $M_2$ is {\em normally attracting\/}.  That is, there exists an invariant Whitney sum decomposition, $M_2 \times \mathbb{R}^n = \mathcal{T} M_2 \oplus E$ that satisfies  the following properties:
  \begin{itemize}
    \item
    There is $c_1 > 0$ such that $\lVert DP_{\epsilon}^k(y)|_{{E}_{y}} \rVert \le c_1 \rho^k$ for any $y \in M_2$ and any $k \in \mathbb{N}$.
    \item
    There is $c_2 > 0$ such that
    \begin{equation*}
      \frac{\lVert DP_{\epsilon}^k(y)|_{E_{y}} \rVert} {\minnorm(DP_{\epsilon}^k(y) |_{\mathcal{T}_{y}M_2})} \le c_2 \left( \frac{\rho}{\sigma} \right)^k
    \end{equation*}
    for any $y \in M_2$ and any $k \in \mathbb{N}$.
  \end{itemize}
In the previous statements,   for each  $y \in M_2$, $E_{y}$ stands for the fiber of $E$ over $y$.
\end{lemma}
\begin{proof}
We denote by $\pi_1$ the projection of $\mathbb{R}^n$ on $\spanned\{v\}$ along $W$, and by $\pi_2$ the projection of $\mathbb{R}^n$ on $W$ along $\spanned\{v\}$. We introduce a new norm, $\norm{\cdot}'$, on $\RR^n$ by putting
\begin{equation*}
  \norm{u}' := \norm{\pi_1 u}_{*} + \norm{\pi_2 u}_{**},
\end{equation*}
where $\norm{\cdot}_{*}$ is the norm on $\spanned\{v\}$ such that $\norm{v}_{*} = 1$ and $\norm{\cdot}_{**}$ is a norm on $W$ with the property that the operator norm $\norm{(DP_{\epsilon}(0)|_{W})^{-1}}_{**,**}< \sigma^{-1}$ (for the existence of such a norm, see~\cite[Prop.~2.8]{Hirsch1977}). We have employed the notation
\begin{equation*}
  \norm{(DP_{\epsilon}(0)|_{W})^{-1}}_{**,**} = \max\{\norm{(DP_{\epsilon}(0)|_{W})^{-1}(p)}_{**}: \norm{p}_{**} = 1\}.
\end{equation*}
The definitions of $\norm{\cdot}_{*,**}$,  $\norm{\cdot}_{**,*}$, and $\norm{\cdot}_{*,*}$ are analogous.

Next we construct an invariant subbundle $E$ as follows:  the fiber $E_{y}$ is given by $\spanned\{v + w(y)\}$ with  $w \colon M_2 \to W$ a continuous map that satisfies $w(0) = 0$ and
\begin{equation*}
   v + w(y) = \frac{DP_{\epsilon}^{-1}(P_{\epsilon}(y))(v + w(P_{\epsilon}(y)))}{\norm{\pi_1 \, DP_{\epsilon}^{-1}(P_{\epsilon}(y))(v + w(P_{\epsilon}(y)))}_{*}}, \quad \forall{\, y \in M_2}.
\end{equation*}
Let us prove the existence of $w$. We define
\begin{equation*}
  \mathcal{K} = \{z \colon M_2 \to W  \text{ continuous} : z(0) = 0\; \text{ and } \norm{z(y)}_{**} \le 1, \; \forall y \in M_2 \}
\end{equation*}
endowed with the metric
\begin{equation*}
  d(z_1, z_2) := \sup\{\, \norm{z_1(y) - z_2(y)}_{**}: y \in M_2 \, \}.
\end{equation*}
Notice that $(\mathcal{K}, d)$ is a complete metric space.  We prove that we can choose $\epsilon > 0$  so that the operator $\mathcal{S}$ defined on $\mathcal{K}$ by the formula
\begin{equation*}
  \mathcal{S}(z)(y) := \frac{DP_{\epsilon}^{-1}(P_{\epsilon}(y))(v + z(P_{\epsilon}(y)))}{\norm{\pi_1 \, DP_{\epsilon}^{-1}(P_{\epsilon}(y))(v + z(P_{\epsilon}(y)))}_{*}} - v, \quad y \in M_2,
\end{equation*}
maps $\mathcal{K}$ into itself and is a contraction. Thus the unique fixed point of $\mathcal{S}$ determines the function $w$. For each $y \in M_2$, we put
  \begin{equation*}
    \begin{aligned}
    A(y) &:= \pi_1 \, DP_{\epsilon}^{-1}(P_{\epsilon}(y))|_{\spanned\{v\}},
    \\
    B(y) &:= \pi_1 \, DP_{\epsilon}^{-1}(P_{\epsilon}(y))|_{W},
    \\
    C(y) &:= \pi_2 \, DP_{\epsilon}^{-1}(P_{\epsilon}(y))|_{\spanned\{v\}},
    \\
    D(y) &:= \pi_2 \, DP_{\epsilon}^{-1}(P_{\epsilon}(y))|_{W}.
    \end{aligned}
  \end{equation*}
In other words, the matrix  $DP_{\epsilon}^{-1}(P_{\epsilon}(y))$ in the decomposition $\RR^n = \spanned\{v\} \oplus W$ has the form
\begin{equation*}
  \begin{bmatrix}
    A(y) & B(y)
    \\
    C(y) & D(y)
  \end{bmatrix}.
\end{equation*}
For any $z \in \mathcal{K}$, we have
\begin{equation*}
  \begin{bmatrix}
    v
    \\
    \mathcal{S}(z)(y)
  \end{bmatrix}
  = \frac{1}{\norm{A(y)v + B(y)  z(P_{\epsilon}(y))}_{*}}
  \begin{bmatrix}
    A(y) & B(y)
    \\
    C(y) & D(y)
  \end{bmatrix}
  \begin{bmatrix}
    v
    \\
     z(P_{\epsilon}(y))
  \end{bmatrix},~
  y \in M_2.
\end{equation*}
As a consequence,
\begin{equation*}
  \begin{aligned}
  \norm{\mathcal{S}(z)(y)}_{**} & = \frac{\norm{C(y)v + D(y) z(P_{\epsilon}(y))}_{**}} {\norm{A(y)v + B(y)  z(P_{\epsilon}(y))}_{*}}
  \\
  & \le \frac{\norm{C(y)}_{*,**}  + \norm{D(y)}_{**,**}  \norm{z(P_{\epsilon}(y))}_{**}}
  {\norm{A(y)}_{*,*} - \norm{B(y)}_{**,*}  \norm{z(P_{\epsilon}(y))}_{**}}.
  \end{aligned}
\end{equation*}
We know that $\norm{A(0)}_{*,*} > \rho^{-1}$, $\norm{D(0)}_{**,**} < \sigma^{-1}$ and $\norm{B(0)}_{**,*} =  \norm{C(0)}_{*,**} = 0$.  Thus, for $\delta > 0$, one can take $\epsilon > 0$ in the construction of $P_{\epsilon}$ so that $\norm{A(y)}_{*,*} > \rho^{-1}$, $\norm{D(y)}_{**,**} < \sigma^{-1}$, $\norm{B(y)}_{**,*} < \delta$ and $\norm{C(y)}_{*,**} < \delta$ for all $y \in M_2$. This implies that
\begin{equation*}
  \norm{\mathcal{S}(z)(y)}_{**} \le \frac{\sigma^{-1}+ {\delta}}{\rho^{-1} - {\delta}}.
\end{equation*}
 Using that $\frac{\sigma^{-1}}{\rho^{-1}} < 1$, we obtain that for $\delta > 0$ small enough, the inequality
\begin{equation}\label{eq:inv-constr-1}
 \frac{\sigma^{-1}+ {\delta}}{\rho^{-1} - {\delta}} \le 1
\end{equation}
is satisfied.  Thus $\mathcal{S}$ maps $\mathcal{K}$ to $\mathcal{K}$.  Our next task is to study when $\mathcal{S}$ is a contraction depending on $\delta$.  For $z_1, z_2 \in \mathcal{K}$ and $y \in M_2$, we write
\begin{align*}
  &  \mathcal{S}(z_1)(y) - \mathcal{S}(z_2)(y) = \frac{D(y)( z_1(P_{\epsilon}(y)) -  z_2(P_{\epsilon}(y)))} {\norm{A(y)v + B(y)  z_1(P_{\epsilon}(y))}_{*}}
  \\
  {} & +
  \left( \frac{1}{\norm{A(y)v + B(y)  z_1(P_{\epsilon}(y))}_{*}} - \frac{1}{\norm{A(y)v + B(y)  z_2(P_{\epsilon}(y))}_{*}}
  \right) (C(y)v + D(y) z_2(P_{\epsilon}(y))).
\end{align*}
The $\norm{\cdot}_{**}$-norm of the first summand is bounded above by
\begin{equation*}
  \frac{\sigma^{-1}}{\rho^{-1} - \delta} \norm{z_1(P_{\epsilon}(y)) -  z_2(P_{\epsilon}(y))}_{**},
\end{equation*}
and the $\norm{\cdot}_{**}$-norm of the second summand is bounded above by
\begin{multline*}
  \frac{\abs{\norm{A(y)v + B(y)  z_2(P_{\epsilon}(y))}_{*} - \norm{A(y)v + B(y)  z_1(P_{\epsilon}(y))}_{*}}} {\norm{A(y)v + B(y) z_1(P_{\epsilon}(y))}_{*} \, \norm{A(y)v + B(y)  z_2(P_{\epsilon}(y))}_{*}} \norm{C(y) v + D(y)  z_2(P_{\epsilon}(y))}_{**}
  \\
  \le \frac{\norm{B(y)  (z_2(P_{\epsilon}(y)) -  z_1(P_{\epsilon}(y)))}_{*}}
  {\norm{A(y)v + B(y)  z_1(P_{\epsilon}(y))}_{*} \, \norm{A(y)v + B(y) z_2(P_{\epsilon}(y))}_{*}} \norm{C(y)v + D(y) z_2(P_{\epsilon}(y))}_{**}
  \\
  \le \frac{\delta (\delta + \sigma^{-1})}{(\rho^{-1} - \delta)^2} \norm{ z_1(P_{\epsilon}(y)) -  z_2(P_{\epsilon}(y))}_{**}.
\end{multline*}
We need to have
\begin{equation}
  \label{eq:inv-constr-2}
  \frac{\sigma^{-1}}{\rho^{-1} - \delta} + \delta \frac{\delta + \sigma^{-1}}{(\rho^{-1} - \delta)^2} < 1.
\end{equation}
 Finally we  choose $\delta>0$ sufficiently small to guarantee the inequalities \eqref{eq:inv-constr-1} and~\eqref{eq:inv-constr-2}.
\end{proof}

\begin{proposition}
\label{prop:aux2}
  There exists a foliation $\mathcal{L} = \{L_{y}\}_{y \in M_2}$ of $\mathbb{R}^n$ given by $C^1$ embedded one-dimensional manifolds $L_{y}$, such that the embeddings depend continuously on $y \in M_2$ in the $C^1$-topology. Moreover,  the following properties are satisfied.
  \begin{itemize}
    \item
    The foliation $\mathcal{L}$ is \emph{invariant}: for each $y \in M_2$, $P_{\epsilon}(L_{y}) = L_{P_{\epsilon}(y)}$.
    \item  For any $y \in M_2$, the tangent space of $L_{y}$ at $y$ is $ E_{y} $.
    \item
    For each $y \in M_2$, the leaf $L_{y}$ is characterized as the set of points $\xi \in \RR^n$ for which
      \begin{equation}
      \label{eq:foliation}
      \frac{\lVert P_{\epsilon}^k(\xi) - P_{\epsilon}^k(y) \rVert} {\minnorm(DP_{\epsilon}^k(y) |_{\mathcal{T}_{y}M_2})} \to 0 \quad \text{as } k \to \infty.
      \end{equation}
    \item
     For each $y \in M_2$, the leaf $L_{y}$ is characterized as the set of points $\xi \in \RR^n$ for which there exists $c = c(x, \xi) > 0$ such that
    \begin{equation}
      \label{eq:foliation-2}
      \lVert P_{\epsilon}^k(\xi) - P_{\epsilon}^k(y) \rVert \le c \rho^k
    \end{equation}
    for all $k \in \mathbb{N}$.
    \end{itemize}
    \end{proposition}
\begin{proof}
  See Theorem~5.5 and Corollary 5.6 in \cite{Hirsch1977}.
\end{proof}

In particular, it follows from the characterization given in~\eqref{eq:foliation-2} that $L_{0} = M_1$. \newline
Now we have all the ingredients to prove the main result of this section.

\medskip

\begin{proof}[Proof of Theorem~\ref{thm:foliation}]
 The $C^1$ embeddings of open intervals that define the foliation $\mathcal{L}$ depend continuously on $y \in M_2$. Therefore, we can write
\begin{equation*}
    L_y = \{\, E(y, s): s \in (s_y^{\mathrm{min}}, s_y^{\mathrm{max}}) \,\}, \quad y \in M_2,
\end{equation*}
where
\begin{equation*}
    E \colon \bigcup\limits_{y \in M_2} \{y\} \times (s_y^{\mathrm{min}}, s_y^{\mathrm{max}}) \to \RR^n
\end{equation*}
is a $C^1$ embedding, with
\begin{gather*}
  M_2 \ni y \mapsto s_y^{\mathrm{min}} \in (-\infty, 0) \\
\intertext{and}
  M_2 \ni y \mapsto s_y^{\mathrm{max}} \in (0, \infty)
\end{gather*}
continuous functions. These maps have the following properties:
\begin{itemize}
    \item For each $y \in M_2$, $E(y, 0) = y$.
    \item For each $y \in M_2$ and $s \in (s_y^{\mathrm{min}}, s_y^{\mathrm{max}})$,
    \begin{equation*}
          \left\lVert \frac{\partial E}{\partial s}(y,s) \right\rVert = 1.
    \end{equation*}
\end{itemize}
We say that $U$ is a \emph{nice neighborhood} of $0$ if $U = E(Z \times (-\delta, \delta))$, where $Z \subset M_2$ is an open disk containing $0$ and $\delta > 0$.  We always assume that the embedding $E$ can be extended to an embedding $\overline{E}$ of (the manifold-with-corners) $\overline{Z} \times [-\delta, \delta]$  where the closure $\overline{Z}$ is a closed disk contained in $M_2$. Moreover $\overline{E}(\overline{Z} \times [-\delta, \delta]) = \overline{U}$.  Such a $\overline{U}$ will be called a \emph{closed nice neighborhood} of $0$. Notice that we can find a neighborhood base of $\RR^n$ at $0$ consisting of nice neighborhoods.

Let $U_{(1)}$ be a nice neighborhood of $0$ such that $\overline{U_{(1)}} \subset V$. The next facts  follow from Propositions~\ref{prop:aux1} and~\ref{prop:aux2} and Lemma~\ref{lm:P-modified}:

\begin{itemize}
  \item $M_1 \cap U_{(1)}$ and $M_2 \cap U_{(1)}$ are locally invariant.
  \item $M_1 \cap U_{(1)}$ is tangent at $0$ to $v$.
  \item $M_2 \cap U_{(1)}$ is tangent at $0$ to $W$.
  \item For any $y \in M_2 \cap U_{(1)}$ and any $\xi \in L_y$, if $P(y), P(\xi) \in U_{(1)}$, then
      \begin{equation*}
        P(\xi) \in L_{P(y)} \cap U_{(1)}.
      \end{equation*}
      Analogously,  if $P^{-1}(y), P^{-1}(\xi) \in U_{(1)}$, then
      \begin{equation*}
        P^{-1}(\xi) \in L_{P^{-1}(y)} \cap U_{(1)}.
      \end{equation*}
\end{itemize}

In the sequel, we frequently make neighborhoods smaller.  In order not to overburden the exposition with notation, we write $M_1$, $M_2$, $L_y$, etc., instead~of $M_1 \cap U_{(1)}$, $M_2 \cap U_{(1)}$, $L_y \cap U_{(1)}$, etc.

\medskip
For $\xi \in \overline{U}$ with $\overline{U}$ a closed nice neighborhood  of $0$, we write
  \begin{equation*}
    v(\xi) := \frac{\partial E}{\partial s}(y,s),
  \end{equation*}
where $s \in [-\delta, \delta]$ and $y \in \overline{Z}$ are chosen so that $\xi = \overline{E}(y, s)$.  Using that $\overline{E}$ is a homeomorphism onto its image, $v(\xi)$ is well defined. Furthermore,  we have that $v(0) = v/\norm{v}$.

By taking a nice neighborhood $U_{(2)} \subset U_{(1)}$, we can assume that $v(\xi) \gg 0$ for all $\xi \in U_{(2)}$.
For $\xi \in U_{(2)}$, we put
  \begin{equation*}
    \begin{aligned}
    A(\xi) &:= \pi_1 \circ DP(\xi)|_{\spanned\{v\}},
    \\
    B(\xi) &:= \pi_1 \circ DP(\xi)|_{W},
    \\
    C(\xi) &:= \pi_2 \circ DP(\xi)|_{\spanned\{v\}},
    \\
    D(\xi) &:= \pi_2 \circ DP(\xi)|_{W}.
    \end{aligned}
  \end{equation*}
We have, for any nonzero $u \in \mathbb{R}^n$,
  \begin{gather*}
    \lVert DP(\xi) u \rVert \le \Bigl(\bigl(\lVert A(\xi) \rVert + \lVert C(\xi) \rVert \bigr) \tfrac{\lVert \pi_1 u \rVert}{\lVert u \rVert} + \bigl(\lVert B(\xi) \rVert + \lVert D(\xi) \rVert\bigr)\tfrac{\lVert \pi_2 u \rVert}{ \lVert u \rVert} \Bigr) \lVert u \rVert
    \\
    \le \Bigl(\bigl(\lVert A(\xi) \rVert + \lVert C(\xi) \rVert \bigr) \bigl(1 + \tfrac{\lVert \pi_2 u \rVert}{\lVert u \rVert} \bigr) + \bigl(\lVert B(\xi) \rVert + \lVert D(\xi) \rVert\bigr)\tfrac{\lVert \pi_2 u \rVert}{ \lVert u \rVert} \Bigr) \lVert  u \rVert.
  \end{gather*}
We know that $\lVert A(0) \rVert < \rho$ and $\lVert B(0) \rVert = \lVert C(0) \rVert = 0$. Hence, we can find  a nice neighborhood $U_{(3)} \subset U_{(2)}$ and a constant $\kappa > 0$ with the following property: for any $\xi \in U_{(3)}$ and any nonzero $u \in \mathbb{R}^n$ with $\lVert \pi_2 u \rVert/\lVert u \rVert \le \kappa$, we have that
  \begin{equation}
  \label{eq:contr}
    \lVert DP(\xi) u \rVert < \rho \lVert u \rVert.
  \end{equation}
Next we take a convex neighborhood $U_{(4)} \subset U_{(3)}$ so that $\lVert \pi_2 v(\xi) \rVert/\lVert v(\xi) \rVert \le \kappa$
 for all $\xi \in U_{(4)}$. This can be done because $\lVert \pi_2 v(0) \rVert/\lVert v(0) \rVert = 0$ and $v(\xi)$ depends continuously on $\xi$. Now we claim that
 \begin{equation*}
   \lVert \pi_2 ({\xi'} - {\xi''}) \rVert/\lVert {\xi'} - {\xi''} \rVert \le \kappa
 \end{equation*}
 for any $y \in M_2 \cap U_{(4)}$ and any $\xi', \xi'' \in L_y \cap U_{(4)}$, $\xi' \ne \xi''$.
Indeed, assume for definiteness' sake that $\xi' = E(y,s')$ and $\xi'' = E(y,s'')$ for some $s' > s''$.  Then
   \begin{equation*}
   \xi'  - \xi'' = \int\limits_{s''}^{s'} v(E(y,\tau)) \, d\tau.
  \end{equation*}
We have $v(E(y,\tau)) \gg 0$, and hence $\pi_1 v(E(y,\tau)) = \alpha(\tau) v$ for some $\alpha(\tau) > 0$.  As
\begin{equation*}
  \lVert \pi_1 v(E(y,\tau)) \rVert \ge \lVert v(E(y,\tau)) \rVert - \lVert \pi_2 v(E(y,\tau)) \rVert  \null \ge   (1-\kappa)v(E(y,\tau)),
\end{equation*}
 one has $\alpha(\tau) \ge 1 - \kappa$, for all $\tau \in [{s'', s'}]$.  Consequently,
  \begin{equation*}
    \pi_1 ({\xi'}  - {\xi''}) = \int\limits_{{s''}}^{{s'}} \pi_1 v(E(y,\tau)) \, d\tau = \Bigl( \int\limits_{{s''}}^{{s'}} \alpha(\tau) \, d\tau \Bigr) v \ge (1 - \kappa) {(s' - s'')} v.
  \end{equation*}
Since $\norm{{\xi'}  - {\xi''}} \le {s' - s''}$, we deduce that $\norm{\pi_1 ({\xi'}  - {\xi''})}/\norm{{\xi'}  - {\xi''}} \ge 1 - \kappa$.  Therefore
\begin{equation*}
 \norm{\pi_2 ({\xi'}  - {\xi''})}/\norm{{\xi'}  - {\xi''}} \le \kappa.
\end{equation*}
As
  \begin{equation*}
    P(\xi') - P(\xi'') = \int\limits_{0}^{1} DP(\xi'' + \tau ({\xi'}  - {\xi''})) ({\xi'}  - {\xi''}) \, d\tau
  \end{equation*}
  for $y \in M_2 \cap U_{(4)}$ and $\xi', \xi'' \in L_y \cap U_{(4)}$ with $P(\xi'), P(\xi'')  \in U_{(4)}$, the above equality  and~\eqref{eq:contr} imply that
  \begin{equation*}
    \norm{{P(\xi')} - {P(\xi'')}} \le \int\limits_{0}^{1} \norm{DP(\xi'' + \tau ({\xi'}  - {\xi''})) ({\xi'}  - {\xi''})} \, d\tau \leq \rho \norm{{\xi'}  - {\xi''}}.
  \end{equation*}
By taking a possibly smaller nice neighborhood $U_{(5)} \subset U_{ (4)}$, we can assume that
  \begin{equation*}
    \norm{E(y,\delta) - y} \in (\rho \delta, \delta] \quad \text{and} \quad \norm{E(y,-\delta) - y} \in (\rho \delta, \delta]
  \end{equation*}
for all $y \in M_2 \cap U_{(5)}$.  Thus, if $y \in M_2 \cap U_{(5)}$ with $P(y) \in M_2 \cap U_{(5)}$, we have that $P(L_y) \subset L_{P(y)}$.

It remains to prove \eqref{eq:M2}. Noticing that the spectral radius of the restriction $DP^{-1}(0)|_W$ is $\nu^{-1}$, we deduce
that there is $l \in \mathbb{N}$ such that $\norm{DP^{-l}(0)|_W} < \sigma^{-l}$. In a manner similar to that used before, we can prove that there exists a neighborhood $U_{(6)} \subset U_{(5)}$ such that
\begin{equation*}
  \norm{DP^{-l}(y) u} \leq \sigma^{-l} \norm{u}
\end{equation*}
 for any $y \in U_{(6)}$ and any $u \in \mathbb{R}^n$ whose direction is sufficiently close to $W$. Now we can take a convex neighborhood $U_{(7)}\subset U_{(6)}$ that satisfies
  \begin{equation*}
    \norm{P^{-l}(y') - P^{-l}(y'')} \le \int\limits_{0}^{1} \norm{DP^{-l}(y'' + \tau (y'  - y'')) (y'  - y'')} \, d\tau \leq \sigma^{-l} \norm{y'  - y''}
  \end{equation*}
  for any $y', y'' \in M_2 \cap U_{(7)}$.
Let $U = U_{(7)}$, which is the desired neighborhood. Thus, we have completed the proof.
\end{proof}

It should be remarked that, in view of the characterization given in Propositions~\ref{prop:aux1} and~\ref{prop:aux2}, a one\nobreakdash-\hspace{0pt}dimensional manifold $M_1$ is unique.  In~particular, it does not depend on the choice of $\rho \in (\mu, \min \{1, \nu\})$ or of the extension $P_{\epsilon}$.  On the other hand, a one\nobreakdash-\hspace{0pt}codimensional manifold $M_2$ depends, in~general, on $\rho$ and the extension $P_{\epsilon}$.  For conditions guaranteeing the uniqueness of $M_2$ see Section~\ref{inv-maniflod}.

\begin{definition}
We say that $M_1$ and $M_{2}$ as defined in Theorem~\ref{thm:foliation} are the $\rho$\nobreakdash-\hspace{0pt}pseudo locally stable manifold and a $\rho$\nobreakdash-\hspace{0pt}pseudo locally unstable manifold respectively.  If $\nu = 1$, then $M_2$ is a \textup{(}local\textup{)} center-unstable manifold at $q$. If $\nu > 1$, then $M_2$ is the \textup{(}local\textup{)} unstable manifold at $q$.
\end{definition}

\section{Linearization, invariant manifolds and tangent cones on the carrying simplex}\label{inv-maniflod}
In this section, we assume, without further mention, that $T \colon \mathbb{R}^{n}_{+} \to \mathbb{R}^{n}_{+}$ is a map of class $C^{1}$ that admits a carrying simplex $S$. Moreover, $T$ has a fixed point $q \in \inte{\mathbb{R}^{n}_{+}}$ so that all the entries of $(DT(q))^{-1}$ are strictly positive and the eigenvalue of $DT(q)$ with the smallest modulus $\mu$ satisfies $0 < \mu < 1$ (condition (C1) in Section \ref{sec:sec-inv-fol}).  We recall that a map $T \colon \mathbb{R}^{n}_{+} \to \mathbb{R}^{n}_{+} $ is of class $C^k$  if there are an open set $ \tilde{U} \subset \mathbb{R}^{n}$ with $\mathbb{R}^{n}_{+} \subset \tilde{U}$, and a $C^k$ map $\tilde{T} \colon \tilde{U} \to \mathbb{R}^{n}$ so that  $\tilde{T}|_{\mathbb{R}^{n}_{+}} = T$.

\subsection{Linearization and invariant manifolds}
The first main result of this section guarantees the conjugacy between the restriction of the map to the carrying simplex in a neighborhood of an interior fixed point and the restriction to its pseudo-unstable manifold.
\begin{theorem}\label{tconjugacy}
  There exist $U$ a neighborhood of $q$ and a homeomorphism $R \colon S\cap U\to M_2$ so that
 \begin{equation}\label{equ:conjugacy}
(T|_{M_2}) \circ R = R \circ (T|_{S\cap U}),
\end{equation}
where $M_{2}$ is defined in Theorem \ref{thm:foliation}.
\end{theorem}
 \begin{proof}
 Take $\overline{U}$ a closed nice neighborhood of $q$ given in Theorem \ref{thm:foliation}. Denote by $\Pi$ the map that assigns to $\xi \in \overline{U}$ the point $y \in M_2$ such that $\xi \in L_{y}$. We observe that  $\Pi$ is a continuous retraction. Now we define
 \begin{equation*}
   R := \Pi|_{S \cap \overline{U}}.
 \end{equation*}
 Since, by (H1), no two points in $S$ are ordered by $<$, $R$ is an injective map of the compact metric space $S \cap \overline{U}$, hence is  a homeomorphism onto its domain.  Moreover, by Theorem~\ref{thm:foliation}(c),
\begin{equation}\label{equ:conjugacy1}
(T|_{M_2}) \circ R = R \circ (T|_{S\cap U}).
\end{equation}
That is, $R$ provides a conjugacy between $T|_{M_2}$ and $T|_{S\cap U}$.
 \end{proof}

From now on, we fix the neighborhood $U$ of $q$ and the homeomorphism $R \colon S\cap U\to M_2$ in Theorem \ref{tconjugacy}, such that
\begin{equation*}
  (T|_{M_2}) \circ R = R \circ (T|_{S\cap U}).
\end{equation*}
The next result is crucial to understanding the unstable manifold on the carrying simplex.
\begin{theorem} \label{thm:smooth-unstable}  Suppose   that there exists a locally invariant $C^1$ submanifold $M' \subset M_2$ so that one of the following conditions is satisfied:
  \begin{enumerate}
    \item[\textup{(i)}] there is a neighborhood base, $\mathcal{V}$, of $q$ in $M'$ so that $T^{-1}(V) \subset V$  for every $V \in \mathcal{V}$,
  \end{enumerate}
  or
  \begin{enumerate}
    \item[\textup{(ii)}] there is a neighborhood base, $\mathcal{V}$, of $q$ in $ R^{-1}(M')$ so that   $T^{-1}(V) \subset V$ for every $V \in \mathcal{V}$.
  \end{enumerate}
  Then, there is a neighborhood $U' \subset U$ of $q$ such that $M' \cap U' \subset S$ \textup{(}$U$ is given in the previous theorem\textup{)}.
\end{theorem}
\noindent To prove this theorem, we need the next result:
\begin{lemma} \label{lm:auxp}
  There exists a constant $\Theta > 0$ such that
  \begin{equation*}
    \norm{\xi - \Pi(\xi)} \le \Theta \norm{q - \Pi(\xi)}
  \end{equation*}
   for all $\xi \in S \cap U$.
\end{lemma}
\begin{proof}
  Assume, by contradiction, that there is no such constant $\Theta$.  Then, for each $m \in \mathbb{N}$, there is $\xi_m \in S \cap U$ such that
  \begin{equation*}
    \norm{\xi_m - \Pi(\xi_m)} > m \norm{q - \Pi(\xi_m)}.
  \end{equation*}
  As the sequence $\{\xi_m\}$ is bounded and $\Pi$ is continuous, we have that $\Pi(\xi_m) \to q$ as $m \to \infty$.  By passing to a subsequence, if~necessary, we can assume that $\xi_m \to \xi$ as $m \to \infty$.  We write
  \begin{equation*}
    \frac{\xi_m - q}{\norm{\xi_m - q}} = \frac{\xi_m - \Pi(\xi_m)}{\norm{\xi_m - q}} + \frac{\Pi(\xi_m) - q}{\norm{\xi_m - q}}.
  \end{equation*}
  Regarding the second term, we have
  \begin{equation*}
    \norm{\xi_m - q} \ge \norm{\xi_m - \Pi(\xi_m)} - \norm{\Pi(\xi_m) - q} > (m - 1) \norm{\Pi(\xi_m) - q}.
  \end{equation*}
  This implies that
  \begin{equation*}
    \frac{\norm{\Pi(\xi_m) - q}}{\norm{\xi_m - q}} < \frac{1}{m - 1}.
  \end{equation*}
  Furthermore, one has
  \begin{multline*}
    (1 -  \tfrac{1}{m-1}) \norm{\xi_m - q} < \norm{\xi_m - q} - \norm{q -  \Pi(\xi_m)} \\
    \le \norm{\xi_m - \Pi(\xi_m)}
    \\
    \le \norm{\xi_m - q} + \norm{q - \Pi(\xi_m)} < (1 + \tfrac{1}{m-1}) \norm{\xi_m - q}.
  \end{multline*}
  Thus
  \begin{equation*}
    \frac{\xi_m - \Pi(\xi_m)}{\norm{\xi_m - q}} - \frac{\xi_m - \Pi(\xi_m)}{\norm{\xi_m - \Pi(\xi_m)}} \to 0
  \end{equation*}
  as $m \to \infty$.  Gathering what we have obtained, we have that either
  \begin{equation*}
    \frac{\xi_m - q}{\norm{\xi_m - q}} \to \frac{\xi - q}{\norm{\xi - q}} \quad \text{as } m \to \infty
  \end{equation*}
  (in the case $\xi \ne q$) or
  \begin{equation*}
    \frac{\xi_m - q}{\norm{\xi_m - q}} \to \pm v
  \end{equation*}
  (in the case $\xi = q$).  In both cases, the limit is in the $\ll$ relation with $0$. This implies that, for $m$ sufficiently large, either $\xi_m \gg q$ or $\xi_m \ll q$. This is a contradiction because no two points in $S$ are related to $\ll$, (see (H1) in Section \ref{sec:sec-notation}).
\end{proof}

\medskip

\begin{proof}[Proof of Theorem~\ref{thm:smooth-unstable}]
  First we notice that   the assumptions (i) and (ii) are equivalent.  Indeed, let (i) be satisfied. Since $R$ is a homeomorphism, $\{R^{-1}(V)\}_{V \in \mathcal{V}}$ is a neighborhood base of $q$ in $R^{-1}(M')$.  It follows from~\eqref{equ:conjugacy} that
  \begin{equation*}
    T^{-1}(R^{-1}(V)) \subset R^{-1}(V)
  \end{equation*}
  for all $R^{-1}(V)$ with $V \in \mathcal{V}$. In an analogous manner, we can prove that (ii) implies (i). Now we prove that the elements of the neighborhood base $\mathcal{V}$ given in (i) are contained in $S$. Pick a point $\xi\in R^{-1}(V)$ with $V$ a member of $\mathcal{V}$. Using that $T^{-1}(R^{-1}(V))\subset R^{-1}(V)$, we have that the negative semitrajectory $\{\ldots, T^{-2}(\xi), T^{-1}(\xi),  \xi\}$ is contained in $R^{-1}(V) \subset S \cap U$.  By Lemma~\ref{lm:auxp},
  \begin{equation*}
    \norm{T^{-j}(\xi) - R(T^{-j}(\xi))} \le \Theta \norm{q - R(T^{-j}(\xi))}
  \end{equation*}
  for all $j = 0, 1, 2, \dots$.  From Theorem~\ref{thm:foliation}(c) and the conjugacy (\ref{equ:conjugacy}), we deduce  that $\norm{\xi - R(\xi)} = \norm{T^j(T^{-j}(\xi)) - T^{j}(R(T^{-j}(\xi)))} \le \rho^{j} \norm{T^{-j}(\xi) - R(T^{-j}(\xi))} \le \rho^{j} \Theta \norm{q - R(T^{-j}(\xi))}$ for all $j$.  Using that $0<\rho<1$, we obtain that $\xi = R(\xi)$. Observe that $\xi\in R^{-1}(V)\subset S\cap U$ and so $R(\xi)\in S\cap U$.
\end{proof}

Next we derive some direct consequences  of Theorems \ref{tconjugacy} and \ref{thm:smooth-unstable}.  The first immediate result is that we can construct the stable and unstable manifolds of $q$ on $S$.
Let $W^s_l(q,T|_{M_2})$ be a $C^1$ local stable manifold and $W^u_l(q,T|_{M_2})$ be the (necessarily unique) $C^1$ local unstable manifold of $q$ in the neighborhood $U$ for $T|_{M_2}$, where $U$ and $M_2$ are given in Theorem \ref{tconjugacy}. Let $M'_l$ be the \textup{(}necessarily unique\textup{)} $C^1$ local unstable manifold of $q$ in the neighborhood $U$ for $T$. By the assumption, we know that $W^u_l(q,T|_{M_2})=M'_l$.

The following theorem gives the precise statement on the construction of the invariant manifolds on $S$.
\begin{theorem}
  \label{thm:inv-local-mani}
  A local stable manifold of $q$ on $S$ given by
  \begin{equation*}
     W^s_l(q,T|_{S}) = R^{-1}(W^s_l(q,T|_{M_2}))
  \end{equation*}
  is a $C^0$ manifold. The local unstable manifold of $q$ on $S$ given by
  \begin{equation*}
    W^u_l(q,T|_{S}) = W^u_l(q,T|_{M_2}) = M'_l
  \end{equation*}
  is a $C^1$ manifold.
\end{theorem}

\noindent As a direct consequence of Theorem \ref{thm:smooth-unstable}, we have the following:

\begin{corollary}
  \label{cor:unstable-global-1}
  The global unstable manifold
  \begin{equation*}
    M'_{g}:= \bigcup_{m=0}^{\infty} T^m(M'_l)
  \end{equation*}
  is contained in $S$.
\end{corollary}

Recall that $q$ is a \emph{hyperbolic} fixed point if there are no eigenvalues of $DT(q)$ having modulus one.  By (C1), this is equivalent to the nonexistence of eigenvalues of $DT(q)|_W$ with modulus one.  The following is a form of the Grobman--Hartman theorem.
\begin{theorem}
  \label{thm:Grobman-Hartman}
  Let $q$ be a hyperbolic fixed point.  Then $T|_S$ is, in a small neighborhood of $q$, topologically conjugate to $DT(q)|_W$.
\end{theorem}
\begin{proof}
  By the classical Grobman--Hartman theorem (see, e.g., Pugh \cite{Pugh1969}), $T|_{M_2}$ is locally topologically conjugate to $DT(q)|_W$.  An application of Theorem~\ref{tconjugacy} concludes the proof.
\end{proof}
In view of the above theorem we can legitimately say that a hyperbolic fixed point $q$ is a \emph{saddle on $S$} if $\nu < 1$ and there is an eigenvalue of $DT(q)|_W$ with modulus greater than one.

\medskip
The conjugacy \eqref{equ:conjugacy} is also useful to compute the index of $q$ on the carrying simplex.
\begin{corollary}\label{coro-index}
 Assume  that $q$ is an isolated fixed point of $T$.  Then
\begin{equation*}
index(T|_S,q) = index(T,q).
\end{equation*}
Moreover, if $1$ is not an eigenvalue of $DT(q)$, then
\begin{equation*}
index(T|_S,q) = index(T,q) = (-1)^m,
\end{equation*}
where $m$ is the sum of the multiplicities of all the eigenvalues of $DT(q)$ which are greater than one.
\end{corollary}
\begin{proof}
From Theorem \ref{tconjugacy}, it is clear that, in a small neighborhood of $q$, there is a topological conjugacy between $T$ and $(T|_{M_{2}},T|_{M_{1}})$, where $M_2$ (resp. $M_1$) is the $\rho$\nobreakdash-\hspace{0pt}pseudo locally $C^1$ unstable (resp. stable) manifold of $q$.  By the multiplicativity of fixed point index (see \cite{Granas2003})
\begin{equation*}
index(T,q) = index(T|_{M_2},q) \cdot index(T|_{M_1},q).
\end{equation*}
Since $0<\mu<1$, we have that $index(T|_{M_1},q) = 1$. It then follows from Theorem \ref{tconjugacy} (see (\ref{equ:conjugacy})) that
\begin{equation*}
index(T,q) = index(T|_{M_2},q) = index(T|_S,q).
\end{equation*}
The last result is immediate.
\end{proof}
\noindent Theorems \ref{tconjugacy} and \ref{thm:smooth-unstable} can be used to give partial responses to  the question posed by M. W. Hirsch in \cite{hirsch1988} on the smoothness of the carrying simplex.
\begin{corollary}
  \label{cor:unstable-global-2}
  If $\nu > 1$,  then $M_{2} = S \cap U$. As a consequence, $S$ is a $C^1$ manifold in a neighborhood of $q$.
\end{corollary}

Following~\cite{Hirsch1977}, we say that $q$ is \emph{Lyapunov unstable on $S$} if for each neighborhood $U_1$ of $q$, there exists a neighborhood $U_2$ of $q$ so that
\begin{equation*}
  U_2 \cap S \subset T^m(U_1 \cap S)
\end{equation*}
for all $m = 0, 1, 2, \dots$  By~\cite[Lemma~5A.2]{Hirsch1977}, this is equivalent to the existence of  a neighborhood base, $\mathcal{V}$, of $q$ in $S$ with the property that  $T^{-1}(V) \subset V$ for any $V \in \mathcal{V}$.  An application of Theorem~\ref{thm:smooth-unstable} gives us the following:
\begin{corollary}
  \label{cor:unstable-global-3}
 If $q$ is Lyapunov unstable on $S$, then $M_{2} = S \cap U$. In particular, $S$ is  a $C^1$ manifold in a neighborhood of $q$.
\end{corollary}

\subsection{Tangent cones}
The tangent cone of the carrying simplex $S$ at a point $\xi \in S$ is defined as
\begin{equation*}
  \mathcal{C}_{\xi}(S) = \{\alpha z : \alpha \ge 0 \text{ and } \frac{\xi_m - \xi}{\norm{\xi_m - \xi}} \to z  \text{ with } \xi_m \subset S \setminus \{\xi\} \text{ a sequence tending to } \xi\}.
\end{equation*}

For each $\xi \in S$, the tangent cone $\mathcal{C}_{\xi}(S)$ is a nontrivial  closed set. That is, it is not  $\{0\}$.
In this subsection, we employ many notation used in Section \ref{sec:sec-inv-fol} such as $\pi_{1}$, $\pi_{2}$, $v$, $W$, and so on.

\begin{lemma}
  \label{lm:aux}
  For a sufficiently small neighborhood $U$ of $q$, there exists a constant $\tilde{\Theta} > 0$ so that
  \begin{equation*}
    \norm{\pi_1(\xi' - \xi'')} \le \tilde{\Theta} \norm{\pi_2(\xi' - \xi'')}
  \end{equation*}
  for all $\xi', \xi'' \in S \cap U$.
\end{lemma}
\begin{proof}
  As the set of those $z$ which are $ \gg 0$ or $\ll 0$ is open, we can take a constant $\tilde{\Theta} > 0$ with the following property: if
  \begin{equation*}
    \norm{\pi_2 (z)} < \frac{1}{\tilde{\Theta}} \norm{\pi_1 (z)}
  \end{equation*}
  for some $z \in \RR^n$,   then $z \gg 0$ or $z \ll 0$.   As $S$ is not ordered by $\ll$, the lemma follows.
\end{proof}
\begin{lemma}
  \label{lm:tangent-proj}
  There exists a neighborhood $U$ of $q$ with the following property: for each $\xi \in S \cap U$ and for each $w \in W$ with $\norm{w} = 1$, there is $z \in \mathcal{C}_{\xi}(S)$ so that $\pi_2 (z) = w$.
\end{lemma}
\begin{proof}
  There are $\delta_1 > 0$ and $\delta_2 > 0$  small enough so that
  \begin{align*}
    & \{\, q + \delta_1 v + \hat{w}: \hat{w} \in W, \norm{\hat{w}} \mathrel< \delta_2 \,\} \subset \mathbb{R}^{n}_{+} \setminus \Gamma,
    \\
    & \{\, q - \delta_1 v + \hat{w} : \hat{w} \in W, \norm{\hat{w}} \mathrel< \delta_2 \,\} \subset \Gamma \setminus S,
  \end{align*}
  where $\Gamma$ is the global attractor of $T$ (see Section \ref{sec:sec-inv-fol} for the precise definition of $v$ and $W$). For each $\eta \in [0, \delta_2)$ and each $\hat{w} \in W$ with $\norm{\hat{w}} = 1$, the segment with endpoints $q + \delta_1 v + \eta \hat{w}$ and $q - \delta_1 v + \eta \hat{w}$ intersects $S$ at precisely one point. We  put
  \begin{equation*}
    U := \{\, q + \beta v + \eta \hat{w}: \beta \in (-\delta_1, \delta_1), \eta \in [0, \delta_2), \hat{w} \in W, \norm{\hat{w}} = 1 \,\}.
  \end{equation*}
  Fix $\xi \in S \cap U$ and $w \in W$ with $\norm{w} = 1$.  By~construction, there is $\delta > 0$ such that for each $\theta \in [0, \delta]$, a point of the form $\xi + \alpha(\theta) v + \theta w$ belongs to $S$.  It suffices to take a limit point of
  \begin{equation*}
    (\alpha(\theta) v + \theta w)/\norm{\alpha(\theta) v + \theta w}
  \end{equation*}
  as $\theta \to 0^{+}$ and multiply it, if~necessary, by a suitable positive scalar to get $z \in  \mathcal{C}_{\xi}(S)$ such that $\pi_2 (z) = w$.
  \end{proof}

A natural  distance between   the tangent cones of two different points $\xi', \xi'' \in S$,
\begin{equation*}
  d(\mathcal{C}_{\xi'}(S), \mathcal{C}_{\xi''}(S))
\end{equation*}
is given by Hausdorff distance between the sets $\{z \in \mathcal{C}_{\xi'}(S): \norm{z} = 1 \,\}$ and $\{z \in \mathcal{C}_{\xi''}(S): \norm{z} = 1 \,\}$.

\begin{theorem}
   \label{thm:tangent}
   The mapping $\xi \mapsto \mathcal{C}_{\xi}(S)$ is continuous at $q$, with  $\mathcal{C}_{q}(S) = W$.
\end{theorem}
\begin{proof}
  In view of Lemma~\ref{lm:tangent-proj}, it suffices to prove that for each $\epsilon > 0$, there is $\delta > 0$ that satisfies the following condition: if $\xi \in S$ and $\norm{\xi - q} < \delta$, then
  \begin{equation*}
    \norm{\pi_1 z}/\norm{\pi_2 z} < \epsilon
  \end{equation*}
  for all  $z \in \mathcal{C}_{\xi}(S) \setminus \{0\}$.  We take  a neighborhood $U_{(0)}$ of $q$ given in Lemma~\ref{lm:aux} for a suitable constant $\tilde{\Theta} > 0$. Let $l \in \NN$ be such that $\norm{DT^{-l}(q)|_{W}} < \sigma^{-l}$.  Given $\epsilon > 0$, we pick $m$  large enough so that
  \begin{equation*}
    \frac{\rho^{ml}}{\sigma^{ml}} < \frac{\epsilon}{4 \tilde{\Theta}},
  \end{equation*}
  (see Section \ref{sec:sec-inv-fol} for the precise definition of $l,\sigma,\rho$, etc).

  For $\xi \in U_{(0)}$, we put
  \begin{equation*}
    \begin{aligned}
    A(\xi) &:= \pi_1 \, DT^{ml}(\xi)|_{\spanned\{v\}},
    \\
    B(\xi) &:= \pi_1 \, DT^{ml}(\xi)|_{W},
    \\
    C(\xi) &:= \pi_2 \, DT^{ml}(\xi)|_{\spanned\{v\}},
    \\
    D(\xi) &:= \pi_2 \, DT^{ml}(\xi)|_{W}.
    \end{aligned}
  \end{equation*}
  In other words, the matrix of $DT^{ml}(\xi)$ in the decomposition $\RR^n = \spanned\{v\} \oplus W$ has the form
  \begin{equation*}
  \begin{bmatrix}
    A(\xi) & B(\xi)
    \\
    C(\xi) & D(\xi)
  \end{bmatrix}.
  \end{equation*}
  We know that $\norm{A(q)} < \rho^{ml}$, $\minnorm(D(q)) > \sigma^{ml}$ and $\norm{B(q)} = \norm{C(q)} = 0$.

  Now we take $U \subset U_{(0)}$  a convex neighborhood of $q$ so that, for all $\xi \in U$, $\norm{A(\xi)} < \rho^{ml}$, $\minnorm(D(\xi)) > \sigma^{ml}$, $\norm{B(\xi)} < \eta$ and $\norm{C(\xi)} < \eta$ with $\eta > 0$  a small number (to be chosen later). We deduce that
  \begin{equation}
   \label{eq:tangent-1}
   \begin{aligned}
   \norm{\pi_1 DT^{ml}(\xi) z} & \le \norm{A(\xi) \pi_1 z} + \norm{B(\xi) \pi_2 z} \le (\rho^{ml} \tilde{\Theta} + \eta) \norm{\pi_2 z},
   \\
   \norm{\pi_2 DT^{ml}(\xi) z} & \ge \norm{D(\xi) \pi_2 z} - \norm{C(\xi) \pi_1 z} \ge (\sigma^{ml} - \eta \tilde{\Theta}) \norm{\pi_2 z},
   \end{aligned}
  \end{equation}
   for all $\xi \in U$ with $T^l(\xi) \in U$ and for all $z \in \RR^n$ with $\norm{\pi_1 z}/\norm{\pi_2 z} < \tilde{\Theta}$.
  Let $\xi', \xi'' \in S \cap U$ be such that $T^{ml}(\xi'), T^{ml}(\xi'') \in S \cap U$.  We have
  \begin{equation*}
    T^{ml}(\xi') - T^{ml}(\xi'') = \int\limits_{0}^{1} DT^{ml}(\xi''+\tau(\xi' - \xi'')) (\xi' - \xi'') \, d\tau.
  \end{equation*}
  Thus,
  \begin{equation*}
    \pi_i T^{ml}(\xi') - \pi_i T^{ml}(\xi'') = \int\limits_{0}^{1} \pi_i DT^{ml}(\xi''+\tau(\xi' - \xi'')) (\xi' - \xi'') \, d\tau
  \end{equation*}
 for $i, j = 1, 2$. By \eqref{eq:tangent-1}, we obtain that
 \begin{equation}
   \label{eq:tangent-2}
   \begin{aligned}
   \norm{\pi_1 T^{ml}(\xi') - \pi_1 T^{ml}(\xi'')} & \le (\rho^{ml} \tilde{\Theta} + \eta) \norm{\pi_2 \xi' - \pi_2 \xi''},
   \\
   \norm{\pi_2 T^{ml}(\xi') - \pi_2 T^{ml}(\xi'')} & \ge (\sigma^{ml} - \eta \tilde{\Theta}) \norm{\pi_2 \xi' - \pi_2 \xi''}.
   \end{aligned}
 \end{equation}
 Now, we take
 \begin{equation*}
   0 < \eta < \min\biggl\{ \frac{\sigma^{ml}}{\tilde{\Theta}}, \frac{\sigma^{ml} \rho^{ml} \tilde{\Theta}}{\sigma^{ml} + 2 \rho^{ml} \tilde{\Theta}^2} \biggr\}.
 \end{equation*}
 We have proved that if $\xi', \xi'' \in S \cap U$, then
 \begin{equation}
   \label{eq:tangent-3}
   \frac{\norm{\pi_1 T^{ml}(\xi') - \pi_1 T^{ml}(\xi'')}}{\norm{\pi_2 T^{ml}(\xi') - \pi_2 T^{ml}(\xi'')}} < \frac{2 \rho^{ml}}{\sigma^{ml}} \tilde{\Theta} < \frac{\epsilon}{2}.
 \end{equation}

 By continuity, for a suitable $\delta > 0$,  $\norm{q - \xi} < \delta$ implies that $\xi \in T^{ml}(U)$.  If, additionally $\xi \in S$, then we have
 \begin{equation*}
   \frac{\norm{\pi_1 z}}{\norm{\pi_2 z}} \le \frac{\epsilon}{2} < \epsilon
 \end{equation*}
 for all $z \in \mathcal{C}_{\xi}(S)\backslash\{0\}$.
\end{proof}

As noted in \cite{Tineo2008}, the paper \cite{Zeeman2003}  takes for granted that  the carrying simplex of a  competitive Lotka--Volterra system of ODEs is tangent to a one\nobreakdash-\hspace{0pt}codimensional invariant subspace (namely, $W$ with our notation).  Our Theorem~\ref{thm:tangent} fills that gap.

\section{Applications to population models}\label{sec:sec-application}
In this section we discuss some applications of the previous results in classical discrete\nobreakdash-\hspace{0pt}time models in population dynamics. We first recall a criterion provided in \cite{Jiang2017} on the existence of a carrying simplex. Let $T \colon \mathbb{R}^{n}_{+} \to \mathbb{R}^{n}_{+}$ be a map of class $C^{1}$ of the form
\begin{equation}\label{pop-model}
  T(x) = (x_1F_1(x),\ldots,x_nF_n(x))
\end{equation}
with $F_{i}(x)  > 0$ for all $i = 1, \ldots, n$ and for all $x = (x_{1},\ldots,x_{n}) \in \mathbb{R}^{n}_{+}$.

\begin{lemma}[\cite{Jiang2017}]\label{Lemma-simplex}
Suppose that the map $T$ satisfies the following properties:
\begin{itemize}
  \item[{\rm (A1)}]
  $\partial F_i(x)/\partial x_j < 0$  for all $x \in \mathbb{R}^n_+$ and $i, j = 1, \ldots, n$.
  \item[{\rm (A2)}]
  $T|_{\mathbb{H}_{\{i\}}^+} \colon \mathbb{H}_{\{i\}}^+\mapsto \mathbb{H}_{\{i\}}^+$ has a fixed point $w_{\{i\}} = w_i e_{\{i\}}$ with $w_i > 0$, $i = 1, \ldots, n$, where $\mathbb{H}_{\{i\}}^+$ is the $i$th positive coordinate axis and $e_{\{i\}}$ is the $i$th vector of the canonical basis.
  \item[{\rm (A3)}]
  $F_i(x) + \sum_{j\in \kappa(x)}x_j\frac{\partial F_i(x)}{\partial x_j} > 0$  or $F_i(x) + \sum_{j\in \kappa(x)}x_i\frac{\partial F_i(x)}{\partial x_j} > 0$  for all $ x \in [0,w] \setminus \{0\}$ and for all $i \in \kappa(x)$, where $\kappa(x) = \{j: x_j > 0 \}$ is the support of $x$, $w = (w_1,\ldots,w_n)$, and the closed order interval $[0,w] = \{x\in \mathbb{R}^{n}_{+}: 0 \leq x_i \leq w_i, i = 1, \ldots, n \}$.
\end{itemize}
Then $T$ admits a carrying simplex $S \subset [0,w]$.
 \end{lemma}

Most discrete-time models of competition induced by the map $T$ of form \eqref{pop-model} satisfy the conditions (A1), (A2) and (A3). Condition (A3) ensures $\det DT(x) > 0$ for all $x \in [0,w]$. (A1) and (A3) imply that $(DT(x)_{\kappa(x)})^{-1}$ has strictly positive entries for all $x \in [0,w]\setminus \{0\}$, and $T$ is competitive in $[0,w]$, where $DT(x)_{\kappa(x)}$ is the submatrix of $DT(x)$ with rows and columns from $\kappa(x)$. In particular, for each interior fixed point $q$ (if exists), all the entries of $(DT(q))^{-1}$ are strictly positive and the eigenvalue of $DT(q)$ with the smallest modulus, say $\mu_q$, satisfies $0 < \mu_q < 1$, that is (C1) in Section \ref{sec:sec-inv-fol} holds. Therefore, we conclude that
\begin{proposition}
All the results in Section \ref{inv-maniflod} hold for the map $T$ of form \eqref{pop-model} which satisfies the conditions \textup{(A1)--(A3)} and has an interior fixed point.
\end{proposition}

Now we will discuss the three\nobreakdash-\hspace{0pt}dimensional (i.e. $n = 3$) maps. We assume the map $T$ given by \eqref{pop-model} satisfies conditions (A1), (A2) and (A3), such that it has a carrying simplex $S$.

\begin{lemma}[\cite{nonlinearity2018}]\label{Lemma-1}
If $T$ has a unique interior fixed point $q \in \inte{\mathbb{R}_+^3}$ such that the eigenvalues of $DT(q)$ are $\mu,\lambda_{1},\lambda_{2}$ with $0 < \mu < \lambda_{1} < 1 < \lambda_{2}$, then the omega limit set of any orbit of $T$ is a connected set consisting of fixed points only.  Moreover, if $T$ has only finitely many fixed points, then any nontrivial orbit of $T$ and any orbit of $(T|_S)^{-1}$ tend to some fixed point of $S$ \textup{(}in this case, we say that $T|_S$ has \emph{trivial dynamics}\textup{)}.
 \end{lemma}

Recalling Theorem \ref{tconjugacy}, the local dynamics of $q$ on $S$ for the map $T$ in Lemma \ref{Lemma-1} is determined by $\lambda_{1}$ and $\lambda_{2}$. Let $\epsilon>0$ such that $\mu_\epsilon = \mu + \epsilon < \lambda_{1}$. Moreover, there exists a homeomorphism $R \colon S \cap U \to M_2$ so that
\begin{equation*}
  (T|_{M_2}) \circ R = R\circ (T|_{S\cap U}),
\end{equation*}
where $U$ is a neighborhood of $q$, and $M_2$ is the $\mu_\epsilon$\nobreakdash-\hspace{0pt}pseudo locally $C^1$ unstable manifold of $q$. Since $\mu_\epsilon < \lambda_{1} < 1 < \lambda_{2}$, there exist a one\nobreakdash-\hspace{0pt}dimensional $C^1$ local stable manifold $W^s_l(q,T|_{M_2})$ of $q$ and a one\nobreakdash-\hspace{0pt}dimensional $C^1$ local unstable manifold  $\beta_l:=W^u_l(q,T|_{M_2})$  of $q$ for $T|_{M_2}$.
Therefore, $q$ is a saddle for $T|_{M_2}$ and hence for $T|_{S}$.  Moreover, by Corollary \ref{coro-index}, one has
\begin{equation*}
  index(T|_S,q) = index(T,q) = -1.
\end{equation*}

Theorem \ref{thm:inv-local-mani} implies that $\alpha_l = R^{-1}(W^s_l(q,T|_{M_2}))$ is a one\nobreakdash-\hspace{0pt}dimensional $C^0$ local stable manifold of $q$ for $T|_S$ and the global stable manifold is given by
\begin{equation*}
  \alpha_g := \bigcup_{k=0}^{\infty}(T|_S)^{-k}(\alpha_l).
\end{equation*}
Moreover, $\beta_g = \bigcup_{k=0}^{\infty}T^{k}(\beta_l)$ is the $C^1$ global unstable manifold of $q$ for $T|_S$.
\newline In particular, the above discussion allows us to prove the following result.

\begin{corollary}\label{Coro-3}
Assume that $T$ satisfies the conditions of Lemma \ref{Lemma-1}  and has four fixed points $\{a_{1},a_{2},r_{1},r_{2}\}$ on the boundary of $S$ with $a_1,a_2$ local attractors and $r_1,r_2$ local repellers.
Then $q$ is a saddle on $S$ so that the global stable manifold $\alpha_g$ of $q$ is a $C^0$ curve joining $r_1$ and $r_2$ and the global unstable manifold $\beta_g$ of $q$ is a $C^1$ curve joining $a_1$ and $a_2$. Moreover, $\alpha_g\cap\beta_g=\{q\}$, and $\alpha_g\cup\beta_g\cup\{a_1,a_2,r_1,r_2\}$ partition $S$ into four invariant components.
\end{corollary}
The phase portrait on $S$ is as shown in either Fig. \ref{fig-global-ds}(a) or  Fig. \ref{fig-global-ds}(b).
\begin{figure}[h!]
    \centering
    \begin{tabular}{ll}
        \subfigure[]{
            \begin{minipage}[b]{0.32\textwidth}
                \centering                \includegraphics[width=\textwidth]{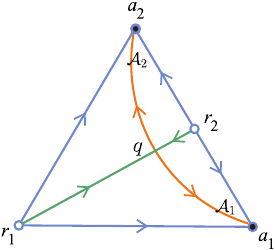}
                \smallskip
            \end{minipage}
        } &
        \subfigure[]{
            \begin{minipage}[b]{0.32\textwidth}
                \centering
                \includegraphics[width=\textwidth]{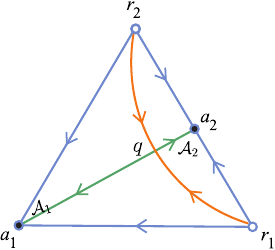}
                \smallskip
            \end{minipage}
        } \\
    \end{tabular}
    \caption{The phase portrait on the carrying simplex $S$, where $\mathcal{A}_i$ denotes the basin of attraction of $a_i$, $i=1,2$. (a) The two attractors $a_1,a_2$ lie on the same edge of $S$. (b) The two repellers  $r_1,r_2$ lie on the same edge of $S$. A fixed point is represented by a closed bullet $\bullet$ if it attracts on $S$, by an open bullet $\circ$ if it repels on $S$, and by the intersection of its stable and unstable manifolds if it is a saddle on $S$.    }
    \label{fig-global-ds}
\end{figure}
\begin{proof}
The first part of the conclusion follows from Theorem \ref{tconjugacy} and Lemma \ref{Lemma-1} immediately.
We now prove the second part of the conclusion. Assume that $\{a_{1},a_{2},r_{1},r_{2}\}$ are four fixed points on the boundary of $S$ with $a_1,a_2$ local attractors and $r_1,r_2$ local repellers. Since $T|_{S}$ has trivial dynamics and $\{r_{1},r_{2}\}$ are local repellers, we have that
\begin{equation}\label{partition}
S = B_{a_{1}} \cup B_{a_{2}} \cup \alpha_g \cup \{r_{1},r_{2}\},
\end{equation}
where $B_{a_{i}}$ is the basis of attraction of $a_{i}$ for the map $T|_{S}$, that is,
\begin{equation*}
  \{p \in S: (T|_{S})^{m}(p) \to a_{i} \text{ as } m \to \infty\}.
\end{equation*}
Notice that $B_{a_{i}}$ is an open set (relative to $S$) and simply connected. Furthermore, $B_{a_{1}}\cap B_{a_{2}}=\emptyset$.  By a simple analysis of the dynamical behavior of $T|_{S}$ on $\partial S$, we have that $\{r_{1},r_{2}\}\in \partial B_{a_{1}}$ and  $\{r_{1},r_{2}\}\in \partial B_{a_{2}}$. Using that $B_{a_{1}}$ and $B_{a_{2}}$ are open and disjoint sets, we can conclude from (\ref{partition}) that
\begin{equation*}
  \partial B_{a_{1}}\cap S\cap \inte{\mathbb{R}^{3}_{+}} \subset \alpha_g,
\end{equation*}
\begin{equation*}
  \partial B_{a_{2}}\cap S\cap \inte{\mathbb{R}^{3}_{+}} \subset \alpha_g,
\end{equation*}
because $\partial B_{a_{1}} \cap B_{a_{2}} = \emptyset$ and $\partial B_{a_{2}} \cap B_{a_{1}} = \emptyset$.  Notice that from (\ref{partition}), $\alpha_g$ is a curve that divides $S$ into two connected components. Moreover, by the previous discussion, it is clear that $\alpha_g$ joins $\{r_{1},r_{2}\}$. By similar arguments, we can prove that $\beta_g$ is a $C^1$ curve joining $a_1$ and $a_2$. The last conclusion is now immediate.
\end{proof}

Corollary \ref{Coro-3} means that if we also know some information on the boundary dynamics, then the global dynamics and the structure of the invariant manifolds on the carrying simplex can be described clearly further. In applications, many models are the same as or similar to the map studied in Corollary \ref{Coro-3}.

\medskip

Consider the following population models:
\begin{itemize}
  \item[I.]
  Leslie--Gower model
  \begin{equation} \label{LG}
    T_i(x)=\frac{(1+r_i)x_i}{1+{ r_i\sum_{j=1}^3}a_{ij}x_j},\  r_i, a_{ij} > 0, \ i, j = 1, 2, 3;
\end{equation}
  \item[II.]
  Atkinson--Allen model
  \begin{equation}\label{equ:AA}
    T_i(x) = \frac{(1+r_i)(1-c_i)x_i}{1+r_i\sum_{j=1}^3a_{ij}x_j} + c_i x_i,\ 0 < c_i < 1, \ r_i, a_{ij} > 0, \ i, j=1,2,3;
\end{equation}
  \item[III.]
  Ricker model
  \begin{equation}\label{equ:Ricker}
    T_i(x) = x_i \exp\Bigl(r_i(1-\sum_{j=1}^3a_{ij}x_j)\Bigr), \ r_i,a_{ij} > 0, \ i, j = 1,2,3.
\end{equation}
\end{itemize}
The conditions (A1)--(A3) hold for the Leslie--Gower model \eqref{LG} and the Atkinson--Allen model \eqref{equ:AA}, so they admit a carrying simplex $S$; see \cite{Jiang2017, Gyllenberg2018a}. The Ricker model \eqref{equ:Ricker} satisfies the conditions (A1)--(A3) if
\begin{equation}\label{Ricker-con-1}
r_i<a_{ii}/\sum_{j=1}^3a_{ij},\text{ or }  r_i<1/\Big(\sum_{j=1}^3\frac{a_{ij}}{a_{jj}}\Big),~i=1,2,3.
\end{equation}
Therefore, it has a carrying simplex $S$ under condition \eqref{Ricker-con-1}; see \cite{Gyllenberg2018b}.

It was shown that there are $33$ stable equivalence classes via an equivalence relation on the boundary dynamics for these three kinds of models; see \cite{Jiang2017, Jiang2016, Gyllenberg2018a, Gyllenberg2018b, nonlinearity2018} for details. According to these papers, the models have a unique interior fixed point $q$ with index $-1$ such that every orbit converges to a fixed point in the classes 19--25. Besides, there are two attracting and two repelling fixed points on the boundary, and the other boundary fixed points (if any) do not attract or repel anything from the interior. That is the models in classes 19--25 are the same as or similar to the map discussed in Corollary \ref{Coro-3}. Thus, we obtain that for these classes the global unstable manifold of $q$ is a $C^1$ curve and the global stable manifold is a $C^0$ curve on the carrying simplex. The dynamical behavior for these systems is given by Corollary \ref{Coro-3}. See Table \ref{table-para-cons} for the precise values of the parameters and the phase portraits on the carrying simplex.  These results solve some open problems suggested in \cite{Jiang2017, Jiang2016, Gyllenberg2018a, Gyllenberg2018b, nonlinearity2018}.

\begin{center}
  \begin{longtable}{c@{\extracolsep{\fill}}c@{\extracolsep{\fill}}c}
\caption{Equivalence classes $19-25$ for models \eqref{LG}, \eqref{equ:AA} and \eqref{equ:Ricker}, where
$\alpha_{ij}:=a_{ii}-a_{ji},~~ \beta_{ij}:=\frac{a_{jj}-a_{ij}}{a_{ii}a_{jj}-a_{ij}a_{ji}}$ for $i,j=1,2,3$ and $i\neq j$. The carrying simplex $S$ with corresponding parameters in each class is given by a representative element in that class, i.e. there exists a permutation of $\{1,2,3\}$ after which parameters of the map satisfy the corresponding inequalities in that class (for the Ricker model \eqref{equ:Ricker}, in addition to the parameter conditions listed for each class, the parameters should also satisfy the additional condition \eqref{Ricker-con-1}). The $s_{\{i\}}$ (resp. $v_{\{i\}}$) denotes a fixed point on the $i$th coordinate axis (resp. plane). The fixed point notation is as in Figure \ref{fig-global-ds}.}\\[-2pt]
        \hline
        \quad Class \qquad & The corresponding parameters   & \qquad Phase portrait in $S$\qquad\\
        \hline
        \endfirsthead
        \caption[]{(continued)}\\
        \hline
        \quad Class \qquad & The Corresponding Parameters   & \qquad Phase Portrait in $S$\qquad\\
        \hline
&&\\
        \endhead
        \hline
        \endfoot
        \endlastfoot
&&\\
19 &
\begin{tabular}{ll} {
 (i)}&{$\alpha_{12}>0, \alpha_{13}>0, \alpha_{21}<0$,} \\
&{$\alpha_{23}<0, \alpha_{31}<0, \alpha_{32}<0$}\\[2pt]
{  (ii)}&{$a_{12}\beta_{23}+a_{13}\beta_{32}<1$
} \end{tabular}
&
    \parbox{2.5cm}{\vspace{2pt}\includegraphics[width=2.5cm]{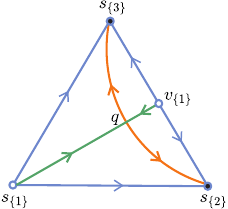}\vspace{2pt}} \\
&&\\
    20 &
\begin{tabular}{ll} {
 (i)}&{$\alpha_{12}<0, \alpha_{13}<0, \alpha_{21}<0$},\\
 &{$\alpha_{23}<0, \alpha_{31}>0, \alpha_{32}<0$}\\[2pt]
{  (ii)}&{$a_{12}\beta_{23}+a_{13}\beta_{32}<1$}\\[2pt]
{  (iii)}&{$a_{31}\beta_{12}+a_{32}\beta_{21}<1$
} \end{tabular}
 &
    \parbox{2.5cm}{\vspace{2pt}\includegraphics[width=2.5cm]{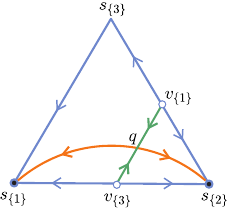}\vspace{2pt}} \\
&&\\
21 &
\begin{tabular}{ll} {
 (i)}&{$\alpha_{12}<0, \alpha_{13}<0, \alpha_{21}<0$},\\
 &{$\alpha_{23}>0, \alpha_{31}<0, \alpha_{32}>0$}\\[2pt]
{  (ii)}&{$a_{12}\beta_{23}+a_{13}\beta_{32}>1$}\\[2pt]
{  (iii)}&{$a_{21}\beta_{13}+a_{23}\beta_{31}<1$}\\[2pt]
{  (iv)}& {$a_{31}\beta_{12}+a_{32}\beta_{21}<1$
} \end{tabular}
&
    \parbox{2.5cm}{\vspace{2pt}\includegraphics[width=2.5cm]{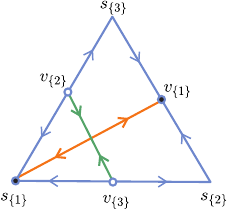}\vspace{2pt}} \\
&&\\
    22 &
\begin{tabular}{ll} {
 (i)}&{$\alpha_{12}>0, \alpha_{13}>0, \alpha_{21}<0$},\\
 &{$\alpha_{23}<0, \alpha_{31}>0, \alpha_{32}<0$}\\[2pt]
{  (ii)}&{$a_{12}\beta_{23}+a_{13}\beta_{32}<1$}\\[2pt]
{  (iii)}& {$a_{21}\beta_{13}+a_{23}\beta_{31}>1$
} \end{tabular}
&
    \parbox{2.5cm}{\vspace{2pt}\includegraphics[width=2.5cm]{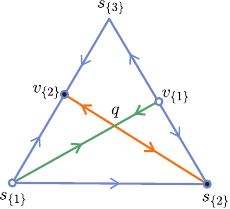}\vspace{2pt}} \\
&&\\
23 &
\begin{tabular}{ll} {
 (i)}&{$\alpha_{12}>0, \alpha_{13}>0, \alpha_{21}>0$},\\
 &{$\alpha_{23}>0, \alpha_{31}<0, \alpha_{32}<0$}\\[2pt]
{  (ii)}&{$a_{31}\beta_{12}+a_{32}\beta_{21}>1$
} \end{tabular}
&
    \parbox{2.5cm}{\vspace{2pt}\includegraphics[width=2.5cm]{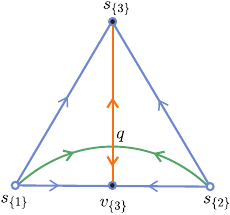}\vspace{2pt}} \\
&&\\
    24 &
\begin{tabular}{ll} {
 (i)}&{$\alpha_{12}>0, \alpha_{13}>0, \alpha_{21}>0$},\\
 &{$\alpha_{23}>0, \alpha_{31}<0, \alpha_{32}>0$}\\[3pt]
{  (ii)}&{$a_{12}\beta_{23}+a_{13}\beta_{32}>1$}\\[3pt]
{  (iii)}&{$a_{31}\beta_{12}+a_{32}\beta_{21}>1$
} \end{tabular}
&
    \parbox{2.5cm}{\vspace{2pt}\includegraphics[width=2.5cm]{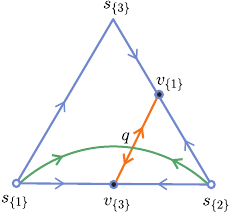}\vspace{2pt}} \\
&&\\
25 &
\begin{tabular}{ll} {
 (i)}&{$\alpha_{12}>0, \alpha_{13}>0, \alpha_{21}>0$},\\
 &{$\alpha_{23}<0, \alpha_{31}>0, \alpha_{32}<0$}\\[2pt]
{  (ii)}&{$a_{12}\beta_{23}+a_{13}\beta_{32}<1$}\\[2pt]
{  (iii)}&{$a_{21}\beta_{13}+a_{23}\beta_{31}>1$}\\[2pt]
{  (iv)}&{$a_{31}\beta_{12}+a_{32}\beta_{21}>1$
} \end{tabular}
&
    \parbox{2.5cm}{\vspace{2pt}\includegraphics[width=2.5cm]{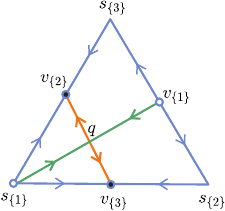}\vspace{2pt}} \\
&&\\
[-8pt]\label{table-para-cons}
\end{longtable}
\end{center}

\section*{Acknowledgement}
The authors are greatly indebted to the referee whose suggestions led to much improvement in the presentation of our results.

The work of J. Mierczy\'nski was supported by project 0401/0155/18.  The work of L. Niu was supported by the Academy of Finland via the Centre of Excellence in Analysis and Dynamics Research (project No. 307333).  The work of A. Ruiz-Herrera was supported by project MTM2017-87697P.

\bibliographystyle{elsarticle-num-names}

\bibliography{Linearization1}

\begin{thebibliography}{44}
\expandafter\ifx\csname natexlab\endcsname\relax\def\natexlab#1{#1}\fi
\providecommand{\url}[1]{\texttt{#1}}
\providecommand{\href}[2]{#2}
\providecommand{\path}[1]{#1}
\providecommand{\DOIprefix}{doi:}
\providecommand{\ArXivprefix}{arXiv:}
\providecommand{\URLprefix}{URL: }
\providecommand{\Pubmedprefix}{pmid:}
\providecommand{\doi}[1]{\href{http://dx.doi.org/#1}{\path{#1}}}
\providecommand{\Pubmed}[1]{\href{pmid:#1}{\path{#1}}}
\providecommand{\bibinfo}[2]{#2}
\ifx\xfnm\relax \def\xfnm[#1]{\unskip,\space#1}\fi
\bibitem[{Hirsch(1988)}]{hirsch1988}
\bibinfo{author}{M.~W. Hirsch},
\newblock \bibinfo{title}{Systems of differential equations which are
  competitive or cooperative: {III}. {C}ompeting species},
\newblock \bibinfo{journal}{Nonlinearity} \bibinfo{volume}{1}
  (\bibinfo{year}{1988}) \bibinfo{pages}{51--71}.
\bibitem[{Smith(1986)}]{smith1986}
\bibinfo{author}{H.~L. Smith},
\newblock \bibinfo{title}{Periodic competitive differential equations and the
  discrete dynamics of competitive maps},
\newblock \bibinfo{journal}{J. Differential Equations} \bibinfo{volume}{64}
  (\bibinfo{year}{1986}) \bibinfo{pages}{165--194}.
\bibitem[{Ortega and Tineo(1998)}]{ortega1998exclusion}
\bibinfo{author}{R.~Ortega}, \bibinfo{author}{A.~Tineo},
\newblock \bibinfo{title}{An exclusion principle for periodic competitive
  systems in three dimensions},
\newblock \bibinfo{journal}{Nonlinear Anal.} \bibinfo{volume}{31}
  (\bibinfo{year}{1998}) \bibinfo{pages}{883--893}.
\bibitem[{Wang and Jiang(2002)}]{wang2002uniqueness}
\bibinfo{author}{Y.~Wang}, \bibinfo{author}{J.~Jiang},
\newblock \bibinfo{title}{Uniqueness and attractivity of the carrying simplex
  for discrete-time competitive dynamical systems},
\newblock \bibinfo{journal}{J. Differential Equations} \bibinfo{volume}{186}
  (\bibinfo{year}{2002}) \bibinfo{pages}{611--632}.
\bibitem[{Diekmann et~al.(2008)Diekmann, Wang, and Yan}]{diekmann2008carrying}
\bibinfo{author}{O.~Diekmann}, \bibinfo{author}{Y.~Wang},
  \bibinfo{author}{P.~Yan},
\newblock \bibinfo{title}{Carrying simplices in discrete competitive systems
  and age-structured semelparous populations},
\newblock \bibinfo{journal}{Discrete Contin. Dyn. Syst.} \bibinfo{volume}{20}
  (\bibinfo{year}{2008}) \bibinfo{pages}{37--52}.
\bibitem[{Hirsch(2008)}]{hirsch2008existence}
\bibinfo{author}{M.~W. Hirsch},
\newblock \bibinfo{title}{On existence and uniqueness of the carrying simplex
  for competitive dynamical systems},
\newblock \bibinfo{journal}{J. Biol. Dyn.} \bibinfo{volume}{2}
  (\bibinfo{year}{2008}) \bibinfo{pages}{169--179}.
\bibitem[{Ruiz-Herrera(2013)}]{Ruiz-Herrera2013}
\bibinfo{author}{A.~Ruiz-Herrera},
\newblock \bibinfo{title}{Exclusion and dominance in discrete population models
  via the carrying simplex},
\newblock \bibinfo{journal}{J. Difference Equ. Appl.} \bibinfo{volume}{19}
  (\bibinfo{year}{2013}) \bibinfo{pages}{96--113}.
\bibitem[{Baigent(2016)}]{Baigent2015}
\bibinfo{author}{S.~Baigent},
\newblock \bibinfo{title}{Convexity of the carrying simplex for discrete-time
  planar competitive {K}olmogorov systems},
\newblock \bibinfo{journal}{J. Difference Equ. Appl.} \bibinfo{volume}{22}
  (\bibinfo{year}{2016}) \bibinfo{pages}{609--622}.
\bibitem[{Jiang et~al.(2016)Jiang, Niu, and Wang}]{jiang2015}
\bibinfo{author}{J.~Jiang}, \bibinfo{author}{L.~Niu},
  \bibinfo{author}{Y.~Wang},
\newblock \bibinfo{title}{On heteroclinic cycles of competitive maps via
  carrying simplices},
\newblock \bibinfo{journal}{J. Math. Biol.} \bibinfo{volume}{72}
  (\bibinfo{year}{2016}) \bibinfo{pages}{939--972}.
\bibitem[{Jiang and Niu(2017)}]{Jiang2017}
\bibinfo{author}{J.~Jiang}, \bibinfo{author}{L.~Niu},
\newblock \bibinfo{title}{On the equivalent classification of three-dimensional
  competitive {L}eslie/{G}ower models via the boundary dynamics on the carrying
  simplex},
\newblock \bibinfo{journal}{J. Math. Biol.} \bibinfo{volume}{74}
  (\bibinfo{year}{2017}) \bibinfo{pages}{1223--1261}.
\bibitem[{Jiang and Niu(2016)}]{Jiang2016}
\bibinfo{author}{J.~Jiang}, \bibinfo{author}{L.~Niu},
\newblock \bibinfo{title}{On the equivalent classification of three-dimensional
  competitive {A}tkinson/{A}llen models relative to the boundary fixed points},
\newblock \bibinfo{journal}{Discrete Contin. Dyn. Syst.} \bibinfo{volume}{36}
  (\bibinfo{year}{2016}) \bibinfo{pages}{217--244}.
\bibitem[{Gyllenberg et~al.(2018)Gyllenberg, Jiang, Niu, and
  Yan}]{Gyllenberg2018a}
\bibinfo{author}{M.~Gyllenberg}, \bibinfo{author}{J.~Jiang},
  \bibinfo{author}{L.~Niu}, \bibinfo{author}{P.~Yan},
\newblock \bibinfo{title}{On the classification of generalized competitive
  {Atkinson--Allen} models via the dynamics on the boundary of the carrying
  simplex},
\newblock \bibinfo{journal}{Discrete Contin. Dyn. Syst.} \bibinfo{volume}{38}
  (\bibinfo{year}{2018}) \bibinfo{pages}{615--650}.
\bibitem[{Gyllenberg et~al.(????)Gyllenberg, Jiang, Niu, and
  Yan}]{Gyllenberg2018b}
\bibinfo{author}{M.~Gyllenberg}, \bibinfo{author}{J.~Jiang},
  \bibinfo{author}{L.~Niu}, \bibinfo{author}{P.~Yan},
\newblock \bibinfo{title}{On the dynamics of multi-species {R}icker models
  admitting a carrying simplex},
\newblock \bibinfo{journal}{submitted}  (????).
\bibitem[{Gyllenberg et~al.(2019)Gyllenberg, Jiang, and Niu}]{Gyllenberg2019}
\bibinfo{author}{M.~Gyllenberg}, \bibinfo{author}{J.~Jiang},
  \bibinfo{author}{L.~Niu},
\newblock \bibinfo{title}{A note on global stability of three-dimensional
  {R}icker models},
\newblock \bibinfo{journal}{J. Difference Equ. Appl.} \bibinfo{volume}{25}
  (\bibinfo{year}{2019}) \bibinfo{pages}{142--150}.
\bibitem[{Niu and Ruiz-Herrera(2018)}]{nonlinearity2018}
\bibinfo{author}{L.~Niu}, \bibinfo{author}{A.~Ruiz-Herrera},
\newblock \bibinfo{title}{Trivial dynamics in discrete-time systems: carrying
  simplex and translation arcs},
\newblock \bibinfo{journal}{Nonlinearity} \bibinfo{volume}{31}
  (\bibinfo{year}{2018}) \bibinfo{pages}{2633--2650}.
\bibitem[{Jiang et~al.(2009)Jiang, Mierczy\'nski, and Wang}]{Jiang2009}
\bibinfo{author}{J.~Jiang}, \bibinfo{author}{J.~Mierczy\'nski},
  \bibinfo{author}{Y.~Wang},
\newblock \bibinfo{title}{Smoothness of the carrying simplex for discrete-time
  competitive dynamical systems: A characterization of neat embedding},
\newblock \bibinfo{journal}{J. Differential Equations} \bibinfo{volume}{246}
  (\bibinfo{year}{2009}) \bibinfo{pages}{1623--1672}.
\bibitem[{Brunovsk\'y(1994)}]{Palo1994}
\bibinfo{author}{P.~Brunovsk\'y},
\newblock \bibinfo{title}{Controlling nonuniqueness of local invariant
  manifolds},
\newblock \bibinfo{journal}{J. Reine Angew. Math.} \bibinfo{volume}{446}
  (\bibinfo{year}{1994}) \bibinfo{pages}{115--135}.
\bibitem[{Mierczy\'nski(1994)}]{Mierczynski1994}
\bibinfo{author}{J.~Mierczy\'nski},
\newblock \bibinfo{title}{The ${C}^1$ property of carrying simplices for a
  class of competitive systems of {ODE}s},
\newblock \bibinfo{journal}{J. Differential Equations} \bibinfo{volume}{111}
  (\bibinfo{year}{1994}) \bibinfo{pages}{385--409}.
\bibitem[{Bena\"{\i}m(1997)}]{Benaim1997}
\bibinfo{author}{M.~Bena\"{\i}m},
\newblock \bibinfo{title}{On invariant hypersurfaces of strongly monotone
  maps},
\newblock \bibinfo{journal}{J. Differential Equations} \bibinfo{volume}{137}
  (\bibinfo{year}{1997}) \bibinfo{pages}{385--409}.
\bibitem[{Mierczy\'nski(1999)}]{Mierczynski1999a}
\bibinfo{author}{J.~Mierczy\'nski},
\newblock \bibinfo{title}{On smoothness of carrying simplices},
\newblock \bibinfo{journal}{Proc. Amer. Math. Soc.} \bibinfo{volume}{127}
  (\bibinfo{year}{1999}) \bibinfo{pages}{543--551}.
\bibitem[{Mierczy\'nski(2018{\natexlab{a}})}]{Mierczynski2018b}
\bibinfo{author}{J.~Mierczy\'nski},
\newblock \bibinfo{title}{The ${C}^1$ property of convex carrying simplices for
  competitive maps},
\newblock \bibinfo{journal}{Ergodic Theory Dynam. Systems}
  (\bibinfo{year}{2018}{\natexlab{a}}) \bibinfo{pages}{DOI:
  10.1017/etds.2018.85}.
\bibitem[{Mierczy\'nski(2018{\natexlab{b}})}]{Mierczynski2018a}
\bibinfo{author}{J.~Mierczy\'nski},
\newblock \bibinfo{title}{The ${C}^1$ property of convex carrying simplices for
  three-dimensional competitive maps},
\newblock \bibinfo{journal}{J. Difference Equ. Appl.} \bibinfo{volume}{24}
  (\bibinfo{year}{2018}{\natexlab{b}}) \bibinfo{pages}{1199--1209}.
\bibitem[{Zeeman and Zeeman(1994)}]{Zeeman1994}
\bibinfo{author}{E.~C. Zeeman}, \bibinfo{author}{M.~L. Zeeman},
\newblock \bibinfo{title}{On the convexity of carrying simplices in competitive
  {L}otka--{V}olterra systems, in {D}ifferential {E}quations, {D}ynamical
  {S}ystems, and {C}ontrol {S}cience},
\newblock \bibinfo{journal}{Lecture Notes in Pure and Appl. Math., 152, Dekker,
  New York}  (\bibinfo{year}{1994}) \bibinfo{pages}{353--364}.
\bibitem[{Zeeman and Zeeman(2002)}]{Zeeman2003}
\bibinfo{author}{E.~C. Zeeman}, \bibinfo{author}{M.~L. Zeeman},
\newblock \bibinfo{title}{From local to global behavior in competitive
  {L}otka--{V}olterra systems},
\newblock \bibinfo{journal}{Trans. Amer. Math. Soc.} \bibinfo{volume}{355}
  (\bibinfo{year}{2002}) \bibinfo{pages}{713--734}.
\bibitem[{Baigent and Hou(2017)}]{Baigent2017}
\bibinfo{author}{S.~Baigent}, \bibinfo{author}{Z.~Hou},
\newblock \bibinfo{title}{Global stability of discrete-time competitive
  population models},
\newblock \bibinfo{journal}{J. Difference Equ. Appl.} \bibinfo{volume}{23}
  (\bibinfo{year}{2017}) \bibinfo{pages}{1378--1396}.
\bibitem[{Baigent(2019)}]{Baigent2019}
\bibinfo{author}{S.~Baigent},
\newblock \bibinfo{title}{Convex geometry of the carrying simplex for the
  {M}ay--{L}eonard map},
\newblock \bibinfo{journal}{Discrete Contin. Dyn. Syst. Ser. B}
  \bibinfo{volume}{24} (\bibinfo{year}{2019}) \bibinfo{pages}{1697--1723}.
\bibitem[{Mierczy\'nski(1999)}]{Mierczynski1999b}
\bibinfo{author}{J.~Mierczy\'nski},
\newblock \bibinfo{title}{Smoothness of carrying simplices for
  three-dimensional competitive systems: A counterexample},
\newblock \bibinfo{journal}{Dynam. Contin. Discrete Impuls. Systems}
  \bibinfo{volume}{6} (\bibinfo{year}{1999}) \bibinfo{pages}{149--154}.
\bibitem[{Pugh(1969)}]{Pugh1969}
\bibinfo{author}{C.~C. Pugh},
\newblock \bibinfo{title}{On a theorem of {P}. {H}artman},
\newblock \bibinfo{journal}{Amer. J. Math.} \bibinfo{volume}{91}
  (\bibinfo{year}{1969}) \bibinfo{pages}{363–367}.
\bibitem[{Hirsch et~al.(1977)Hirsch, Pugh, and Shub}]{Hirsch1977}
\bibinfo{author}{M.~W. Hirsch}, \bibinfo{author}{C.~C. Pugh},
  \bibinfo{author}{M.~Shub}, \bibinfo{title}{Invariant Manifolds, Lecture Notes
  in Mathematics}, volume \bibinfo{volume}{583}, \bibinfo{publisher}{Springer,
  Berlin-New York}, \bibinfo{year}{1977}.
\bibitem[{Quandt(1986)}]{Quandt1986}
\bibinfo{author}{J.~Quandt},
\newblock \bibinfo{title}{On the {H}artman--{G}robman theorem for maps},
\newblock \bibinfo{journal}{J. Differential Equations} \bibinfo{volume}{64}
  (\bibinfo{year}{1986}) \bibinfo{pages}{154--164}.
\bibitem[{Zhang(1993)}]{Zhang1993}
\bibinfo{author}{W.~Zhang},
\newblock \bibinfo{title}{Generalized exponential dichotomies and invariant
  manifolds for differential equations},
\newblock \bibinfo{journal}{Adv. Math.} \bibinfo{volume}{22}
  (\bibinfo{year}{1993}) \bibinfo{pages}{1--45}.
\bibitem[{Bronstein and Kopanskii(1994)}]{Bronstein1994}
\bibinfo{author}{I.~U. Bronstein}, \bibinfo{author}{A.~Y. Kopanskii},
  \bibinfo{title}{Smooth Invariant Manifolds and Normal Forms},
  \bibinfo{publisher}{World Scientific}, \bibinfo{year}{1994}.
\bibitem[{Tan(2000)}]{Tan2000}
\bibinfo{author}{B.~Tan},
\newblock \bibinfo{title}{$\sigma$-{H}\"{o}lder continuous linearization near
  hyperbolic fixed points in $\mathbb{R}^n$},
\newblock \bibinfo{journal}{J. Differential Equations} \bibinfo{volume}{162}
  (\bibinfo{year}{2000}) \bibinfo{pages}{251--269}.
\bibitem[{Nipp and Stoffer(2013)}]{Nipp2013}
\bibinfo{author}{K.~Nipp}, \bibinfo{author}{D.~Stoffer},
  \bibinfo{title}{Invariant Manifolds in Discrete and Continuous Dynamical
  Systems}, \bibinfo{publisher}{European Mathematical Society},
  \bibinfo{year}{2013}.
\bibitem[{Kuznetsov(2004)}]{Kuznetsov3}
\bibinfo{author}{Y.~A. Kuznetsov}, \bibinfo{title}{Elements of Applied
  Bifurcation Theory, Third Edition}, \bibinfo{publisher}{Springer-Verlag, New
  York}, \bibinfo{year}{2004}.
\bibitem[{Fenichel(1979)}]{Fenichel1979}
\bibinfo{author}{N.~Fenichel},
\newblock \bibinfo{title}{Geometric singular perturbation theory for ordinary
  differential equations},
\newblock \bibinfo{journal}{J. Differential Equations} \bibinfo{volume}{31}
  (\bibinfo{year}{1979}) \bibinfo{pages}{53--98}.
\bibitem[{Palis and Takens(1993)}]{Palis1993}
\bibinfo{author}{J.~Palis}, \bibinfo{author}{F.~Takens},
  \bibinfo{title}{Hyperbolicity and Sensitive Chaotic Dynamics at Homoclinic
  Bifurcations}, \bibinfo{publisher}{Cambridge University Press, Cambridge},
  \bibinfo{year}{1993}.
\bibitem[{Chow et~al.(1991)Chow, Lin, and Lu}]{Chow1991}
\bibinfo{author}{S.-N. Chow}, \bibinfo{author}{X.-B. Lin},
  \bibinfo{author}{K.~Lu},
\newblock \bibinfo{title}{Smooth invariant foliations in infinite dimensional
  spaces},
\newblock \bibinfo{journal}{J. Differential Equations} \bibinfo{volume}{94}
  (\bibinfo{year}{1991}) \bibinfo{pages}{266--291}.
\bibitem[{Bates et~al.(2000)Bates, Lu, and Zeng}]{Bates2000}
\bibinfo{author}{P.~W. Bates}, \bibinfo{author}{K.~Lu},
  \bibinfo{author}{C.~Zeng},
\newblock \bibinfo{title}{Invariant foliations near normally hyperbolic
  invariant manifolds for semiflows},
\newblock \bibinfo{journal}{Trans. Amer. Math. Soc.} \bibinfo{volume}{352}
  (\bibinfo{year}{2000}) \bibinfo{pages}{4641--4676}.
\bibitem[{Zhang and Zhang(2016{\natexlab{a}})}]{zhang_zhang_2016}
\bibinfo{author}{W.~Zhang}, \bibinfo{author}{W.~Zhang},
\newblock \bibinfo{title}{${\it\alpha}$-{H}\"{o}lder linearization of
  hyperbolic diffeomorphisms with resonance},
\newblock \bibinfo{journal}{Ergodic Theory Dynam. Systems} \bibinfo{volume}{36}
  (\bibinfo{year}{2016}{\natexlab{a}}) \bibinfo{pages}{310--334}.
\bibitem[{Zhang and Zhang(2016{\natexlab{b}})}]{Zhang2016}
\bibinfo{author}{W.~Zhang}, \bibinfo{author}{W.~Zhang},
\newblock \bibinfo{title}{On invariant manifolds and invariant foliations
  without a spectral gap},
\newblock \bibinfo{journal}{Adv. Math.} \bibinfo{volume}{303}
  (\bibinfo{year}{2016}{\natexlab{b}}) \bibinfo{pages}{549--610}.
\bibitem[{Tineo(2008)}]{Tineo2008}
\bibinfo{author}{A.~Tineo},
\newblock \bibinfo{title}{May {L}eonard systems},
\newblock \bibinfo{journal}{Nonlinear Anal. Real World Appl.}
  \bibinfo{volume}{9} (\bibinfo{year}{2008}) \bibinfo{pages}{1612--1618}.
\bibitem[{Hirsch(1994)}]{Hirsch1994}
\bibinfo{author}{M.~W. Hirsch}, \bibinfo{title}{Differential Topology,
  corrected reprint of the 1976 original, Grad. Texts in Math.},
  volume~\bibinfo{volume}{38}, \bibinfo{publisher}{Springer, New York},
  \bibinfo{year}{1994}.
\bibitem[{Granas and Dugundji(2003)}]{Granas2003}
\bibinfo{author}{A.~Granas}, \bibinfo{author}{J.~Dugundji},
  \bibinfo{title}{Fixed Point Theory}, \bibinfo{publisher}{Springer-Verlag, New
  York}, \bibinfo{year}{2003}.

\end{thebibliography}

\end{document}